\documentclass[12pt,a4paper]{amsart}
\makeatletter

\author{Fan Qin}

\email{qin.fan.math@gmail.com}

\theoremstyle{plain}
 \newtheorem{theorem}{Theorem}[section]

\theoremstyle{definition}

\usepackage[utf8]{inputenc}
\usepackage[active]{srcltx}
\usepackage{amsbsy}
\usepackage{amstext}
\usepackage{amsthm}

\providecommand{\tabularnewline}{\\}

\usepackage{amsmath}

\usepackage{graphicx}
\usepackage{subfig}

\usepackage{amssymb}
\usepackage{stmaryrd}
\usepackage{amsthm}
\usepackage[mathscr]{eucal}
\usepackage{amscd}
\usepackage[all]{xy}
\usepackage{url}

\usepackage{aliascnt}
\newaliascnt{Thm}{theorem}

\usepackage{enumerate}

\usepackage{comment}

\usepackage{hyperref}
\hypersetup{%
pdftitle={Preprint},
pdfauthor={},
pdfkeywords={},
bookmarksnumbered,
pdfstartview={FitH},
breaklinks=true,
colorlinks=true,
}%
\usepackage[all]{hypcap} 

\usepackage{breakurl}

\usepackage{color}

\newtheorem{Quest}[Thm]{Question}

\newtheorem*{Def*}{Definition}
\newtheorem*{Thm*}{Theorem}
\newtheorem{Assumption}{Assumption}
\newtheorem*{Conj*}{Conjecture}

\newcommand{\kk}{\Bbbk}
\newcommand{\Z}{\mathbb{Z}}
\newcommand{\N}{\mathbb{N}}
\newcommand{\Q}{\mathbb{Q}}
\newcommand{\C}{\mathbb{C}}
\newcommand{\R}{\mathbb{R}}

\newcommand{\opname}[1]{\operatorname{\mathsf{#1}}}

\renewcommand{\mod}{\opname{mod}}

\newcommand{\pr}{\opname{pr}}

\newcommand{\op}{^{op}}

\newcommand{\Gr}{\opname{Gr}}

\newcommand{\Hom}{\opname{Hom}}

\renewcommand{\deg}{\opname{deg}}

\newcommand{\Hf}{{\frac{1}{2}}}

\newcommand{\Rm}[1]{{\longmapsto}}
\newcommand{\Lm}[1]{{\longmapsfrom}}

\newcommand{\cA}{{\mathcal A}}

\newcommand{\cC}{{\mathcal C}}

\newcommand{\cF}{{\mathcal F}}

\newcommand{\bL}{{\mathbf L}}

\newcommand{\uc}{{\underline{c}}}

\newcommand{\tB}{{\widetilde{B}}}

\newcommand{\tQ}{{\widetilde{Q}}}

\newcommand{\hJ}{{\widehat{J}}}

\newcommand{\can}{L}
\newcommand{\gen}{\mathbb{L}}

\newcommand{\clAlg}{{\cA}}

\newcommand{\diag}{{\delta}}

\usepackage{tikz}
\usetikzlibrary{positioning,shapes,shadows,arrows,snakes}

\pgfdeclarelayer{edgelayer}
\pgfdeclarelayer{nodelayer}
\pgfsetlayers{edgelayer,nodelayer,main}

\tikzstyle{none}=[inner sep=0pt]
\tikzstyle{black box}=[draw=black, fill=black!25]
\tikzstyle{white box}=[draw=black, fill=white]
\tikzstyle{black circle}=[circle,draw=black!50, fill=black!25]
\tikzstyle{red circle}=[circle,draw=red!50, fill=red!25]
\tikzstyle{blue circle}=[circle,draw=blue!50, fill=blue!25]
\tikzstyle{green circle}=[circle,draw=green!50, fill=green!25]
\tikzstyle{yellow circle}=[circle,draw=yellow!50, fill=yellow!25]

\newcommand{\thistheoremname}{}
\newtheorem*{genericthm*}{\thistheoremname}
\newenvironment{namedthm*}[1]
  {\renewcommand{\thistheoremname}{#1}%
   \begin{genericthm*}}
  {\end{genericthm*}}

\renewcommand{\diag}{{d}}

\renewcommand{\can}{{\bL}}

\newcommand{\fv}{\opname{f}}
\newcommand{\ufv}{\opname{uf}}

\numberwithin{equation}{section}
\numberwithin{figure}{section}

\theoremstyle{plain}
 \newtheorem{thm}[theorem]{Theorem}

\theoremstyle{plain}
\newtheorem{conjecture}[theorem]{Conjecture}

\theoremstyle{definition}
\newtheorem{example}[theorem]{Example}

\theoremstyle{definition}
\newtheorem{defn}[theorem]{Definition}
\theoremstyle{remark}
\newtheorem{rem}[theorem]{Remark}
\theoremstyle{plain}
\newtheorem{lem}[theorem]{Lemma}

\tikzstyle{marked}=[inner sep=0.8mm,circle,draw,fill=black!80]

\usepackage{subfig}
\makeatother

\newcommand{\red}{\color{red}}

\newcommand{\alert}{\red}

\newcommand{\Brac}{\opname{Brac}}
\newcommand{\Band}{\opname{Band}}

\renewcommand{\Mc}{M^\circ}

\newcommand{\upClAlg}{\mathcal{U}}

\newcommand{\AVar}{\mathbb{A}}
\newcommand{\XVar}{\mathbb{X}}

\newcommand{\bClAlg}{{\overline{\clAlg}}}
\newcommand{\bUpClAlg}{{\overline{\upClAlg}}}

\newcommand{\cRing}{\mathcal{P}}

\renewcommand{\diag}{d'}

\newcommand{\frg}{{\mathfrak{g}}}

\newcommand{\Sk}{\opname{Sk}}

\newcommand{\col}{\opname{col}}

\newcommand{\Inj}{\mathbf{I}}

\newcommand{\envAlg}{\opname{U_q}}

\newcommand{\ow}{\overrightarrow{w}}

\newtheorem{Definition-Lemma}[Thm]{Definition-Lemma}

\newcommand{\qO}{{\opname{A_q}}}

\newcommand{\opup}{{\mathrm{up}}}

\newcommand{\tropMc}{\mathcal{M}^\circ}

\newcommand{\dCan}{{\mathbf{B}^*}}

\begin{document}



\title{Cluster algebras and their bases}
\maketitle






\begin{abstract}
We give a brief introduction to (upper) cluster algebras and their
quantization using examples. Then we present several important families
of bases for these algebras using topological models. We also discuss
tropical properties of these bases and their relation to representation
theory.

This article is an extended version of the talk given at the 19th
International Conference on Representations of Algebras (ICRA 2020). 
\end{abstract}


\section{Introduction}

\label{sec:intro}

\subsection{Cluster algebras and the dual canonical basis}

\cite{FominZelevinsky02} invented cluster algebras as an algebraic
framework to study the dual canonical basis $\dCan$ \cite{Lusztig90,Lusztig91}\cite{Kashiwara90}
and total positivity \cite{Lusztig96}. These are commutative
algebras with distinguished generators called cluster variables, which
are grouped into overlapping sets called clusters. The monomials of
cluster variables from the same cluster are called the cluster monomials.
See the first example in Section \ref{subsec:A-frist-example}.

\cite{FominZelevinsky02} expected that the coordinate rings of many
interesting varieties arising from a simply-connected semisimple group
$G$ have a natural structure of a cluster algebra. Moreover, they
should possess an analogue of the original dual canonical bases, and
these bases contain all cluster monomials. Therefore, it is a natural
and important question to construct well-behaved bases for cluster
algebras which contain all cluster monomials.

\cite{BerensteinZelevinsky05} introduced quantization for cluster
algebras. As a main motivation of \cite{FominZelevinsky02}, the expected
relation between quantum cluster monomials and the original dual canonical
basis can be formulated as below (\cite[Conjecture 1.1]{Kimura10}).
\begin{conjecture}[Quantization conjecture]
\label{conj:FZ_conj}Given any symmetrizable Kac-Moody algebra $\mathfrak{g}$
and any Weyl group element $w$, up to scalar multiples in $q^{\Hf\Z}$,
the corresponding quantum unipotent subgroup $\qO[N_{-}(w)]$ is a
quantum cluster algebra and its dual canonical basis contains all
quantum cluster monomials.
\end{conjecture}

Here $N_{-}(w)$ is the unipotent subgroup and $\qO[N_{-}(w)]$ the
quantum analog of its ring of functions, see Sections \ref{subsec:cluster_q_gp_example}
and \ref{subsec:From-dual-canonical-to-triangular} for more details. 
\begin{proof}
The cluster structure on $\qO[N_{-}(w)]$ was verified by \cite{GeissLeclercSchroeer10,GeissLeclercSchroeer11}
and \cite{GY13,goodearl2020integral}. 

The dual canonical basis part in Conjecture \ref{conj:FZ_conj} was
verified for acyclic quivers by \cite{KimuraQin14} (based on \cite{HernandezLeclerc09}\cite{Nakajima09}),
for simply-laced semisimple Lie algebras and partially for symmetric
Kac-Moody algebras by \cite{qin2017triangular}, for all symmetric
Kac-Moody algebras by \cite{Kang2018}, and for all cases by \cite{qin2020dual}. 
\end{proof}

\subsection{Well-behaved bases}

The approach in \cite{qin2017triangular} \cite{qin2020dual} for
Conjecture \ref{conj:FZ_conj} is based on showing that the dual canonical
basis provides the common triangular basis proposed in \cite{qin2017triangular},
see Section \ref{sec:triangular_basis} and Theorem \ref{thm:can_triangular_basis}.
This approach heavily depends on the analysis of tropical properties
of (upper) cluster algebras (Section \ref{subsec:Upper-cluster-algebras}
and \ref{sec:Tropical-properties-and-bases}). 

More precisely, the geometric counterpart for a cluster algebra is
called the cluster variety, which is constructed by gluing split tori
corresponding to the clusters. Its ring of functions is called the
upper cluster algebra. The cluster algebra is contained in the upper
cluster algebra and they often agree. Any upper cluster algebra element
can be expressed as a Laurent polynomial ring in the cluster variables
from any chosen cluster, and its Laurent expansion depends on the choice
of the cluster. By tropical properties, we mean how the Laurent degrees
change when we change the cluster. See Section \ref{sec:Tropical-properties-and-bases}
for details.

It is expected by Fock and Goncharov that the upper cluster algebra
possesses a basis with good tropical properties (i.e. parametrized
by the tropical points, see Section \ref{subsec:Tropical-properties}),
see Conjecture \ref{conj:FG}.

In the literature, for studying bases of cluster algebras, we usually
need to add a condition called the full rank assumption (Assumption
\ref{assumption:full_rank}). Under this assumption, there are three
families for an upper cluster algebra that are well-behaved, important
and intensively-studied:
\begin{itemize}
\item The generic basis in the sense of \cite{dupont2011generic}, see Section
\ref{sec:generic_basis}: For cluster algebras arising from quivers
which satisfy the injective-reachable assumption (Assumption \ref{assumption:injective-reachable}),
the existence of the generic basis is proved in \cite{qin2019bases}.
By \cite{GeissLeclercSchroeer10b,GeissLeclercSchroeer10}, this family
of bases includes the dual semi-canonical basis for the coordinate
ring $\C[N_{-}(w)]$ in the sense of Lusztig \cite{Lusztig00}. \cite{geiss2020generic,geiss2020schemes}
showed that this family includes the bangle basis for the Skein algebra
of an unpunctured surface with marked points \cite{musiker2013bases},
see Section \ref{sec:Topological-models}.
\item The common triangular basis in the sense of \cite{qin2019bases}:
This family of bases includes the dual canonical basis by \cite{qin2020dual}
(Theorem \ref{thm:can_triangular_basis}). Moreover, it often includes
the set of the isoclasses of simple modules when a monoidal categorification
is given. See Section \ref{sec:triangular_basis} for more details.
\item The theta basis in the sense of \cite{gross2018canonical}: Many cluster
algebras possess this basis (for example, when the injective-reachable
assumption holds) \cite{gross2018canonical}. It consists of the \emph{theta
functions }constructed in the \emph{scattering diagrams}, which arise
from the geometry of the cluster varieties. This family includes the
greedy basis \cite{lee2014greedy}\cite{cheung2015greedy}. \cite{MandelQin2021}
shows that it includes the bracelet basis for Skein algebras of surfaces
with marked points \cite{musiker2013bases}.
\end{itemize}
All three families of bases above are parametrized by the tropical
points (i.e. meet the expectation of Fock and Goncharov). By \cite[Theorem 1.2.1]{qin2019bases},
there are infinitely many bases with such properties, see Section
\ref{subsec:Bases-with-good-tropical}. Nevertheless, sometimes as
conjectures, all known good bases for cluster algebras in literature
should belong to these three families. See Section \ref{sec:Further-topics}
for a further discussion.

\subsection{Contents}

This article is an extended version of the author's talk given at
the 19th International Conference on Representations of Algebras (ICRA
2020). We try to give a concise introduction to cluster algebras and
their bases, in particular for those arising from quantum groups \cite{GeissLeclercSchroeer10,GeissLeclercSchroeer11}\cite{GY13,goodearl2020integral}
or from surfaces \cite{FominShapiroThurston08}. The constructions
and results we present are mostly from \cite{qin2017triangular,qin2019bases,qin2020dual}\cite{musiker2013bases}\cite{thurston2014positive}.

There has been a wide range of literature on cluster algebras, which
focuses on different problems and is based on different methods. We
omit many important structures and results for cluster algebras. We
refer the reader to \cite{Keller08c} for a good survey covering many
aspects of cluster algebras, to \cite{HernandezLeclerc09}\cite{Kang2018}
for their monoidal categorification and to \cite{gross2013birational}\cite{gross2018canonical}
for cluster varieties and the theta basis.

In Section \ref{sec:A-brief-introduction}, we give a brief introduction
to cluster algebras. We start by giving an example of a cluster algebra
using unitriangular matrices in Section \ref{subsec:A-frist-example}.
This will be our first running example for cluster algebras with a
Lie theoretic background. Then we give general definitions for cluster
algebras, upper cluster algebras, and their quantization. In Section
\ref{subsec:cluster_q_gp_example}, we discuss the quantized version
of the first example, which provides the first example for the quantum
unipotent subgroup $\qO[N_{-}(w)]$, its localization $\qO[N_{-}^{w}]$,
their dual canonical bases and their quantum cluster structure.

Section \ref{sec:Topological-models} is independent. By using topological
models, it aims to provide intuition to readers unfamiliar with cluster
algebras and their bases. We introduce Skein algebras for surfaces
with marked points, which we assume to have no punctures for simplicity.
These algebras are (often) cluster algebras. Following \cite{thurston2014positive},
we present the three important families of bases for these algebras
using curves on surfaces. We give our second running example on the annulus
(Example \ref{example:Annulus}).

In Section \ref{sec:Tropical-properties-and-bases}, we introduce
the dominance orders for Laurent degrees of Laurent polynomials, define
the tropical transformations, and give a simplified definition for
the tropical points. Then we clarify what we mean by good tropical
properties (i.e. parametrized by the tropical points), and how these
properties can be used to study bases.

In Section \ref{sec:triangular_basis}, we define the triangular basis
and the common triangular basis. We relate these bases to the dual
canonical basis for quantum groups and the simple modules in monoidal
categories.

In Section \ref{sec:generic_basis}, we give a brief introduction
to the generic basis, which is constructed from quiver Grassmannians
for quiver representations via the Caldero-Chapoton map. We discuss
the coefficient-free version of our second running example on the annulus
(Example \ref{example:Kronecker}).

In Section \ref{sec:Further-topics}, we give some important and open
questions concerning bases for cluster algebras.

\section{A brief introduction to cluster algebras\label{sec:A-brief-introduction}}

\subsection{A first example for total positivity, cluster algebras and their
bases\label{subsec:A-frist-example}}

Choose the base ring to be $\kk=\R$. Let $G=SL_{3}(\kk)$ denote
the group of the matrices with determinant $1$. It has the following
subgroup (unipotent radical):
\[
N_{-}:=\left\{ g=\left(\begin{array}{ccc}
1 & 0 & 0\\
X_{1} & 1 & 0\\
X_{2} & X_{1}' & 1
\end{array}\right),\text{\ensuremath{\forall X_{1},X_{1}',X_{2}} }\right\} 
\]
The ring of functions $\kk[N_{-}]$ is the polynomial ring $\kk[X_{1},X_{1}',X_{2}]$.
The monomials in $X_{1},X_{1}',X_{2}$ form a $\kk$-basis for $\kk[N_{-}]$,
which is the \textbf{dual PBW basis} for the corresponding quantum
group evaluated at $q=1$.

A matrix $g$ is said to be \textbf{totally non-negative} if all of its minors are non-negative. It has four interesting minors $X_{1},X_{1}',X_{2},X_{3}:=\left|\begin{array}{cc}
X_{1} & 1\\
X_{2} & X_{1}'
\end{array}\right|$ subject to the following algebraic relation
\begin{align*}
X_{1}' & =  X_{1}^{-1}\cdot(X_{3}+X_{2})
\end{align*}
They provide examples of \textbf{cluster variables} and the corresponding
\textbf{exchange relation}. 

We group the four minors into two overlapping subsets called \textbf{clusters}
$\{X_{1},X_{2},X_{3}\}$, $\{X_{1}',X_{2},X_{3}\}$. The monomials
from each subset are called \textbf{cluster monomials}. In this example,
the cluster monomials form a $\kk$-basis for $\kk[N_{-}]$, which
is the \textbf{dual canonical basis} $\dCan$ for the corresponding
quantum group evaluated at $q=1$.

We refer the reader to Fomin's ICM talk \cite{fomin2010total} for
more details on total positivity and cluster algebras.

\subsection{A general definition for cluster algebras}

Choose a base ring $\kk$, which will be $\Z$ for the classical case
or $\Z[v^{\pm}]$ for the quantum case in this paper. Here, $v$ denotes
a formal quantum parameter, and we denote $q=v^{2}$ and $v=q^{\frac{1}{2}}$.

\subsubsection*{A seed}

Let $I$ denote a finite set of vertices endowed with a partition
into unfrozen vertices and frozen vertices $I=I_{\ufv}\sqcup I_{\fv}$.
We further fix a collection of strictly positive integers $d_{i}$,
$i\in I$. Denote the diagonal matrix $D=\opname{diag}(d_{i})_{i\in I_{\ufv}}$.

A seed $t$ consists of indeterminates $X_{i}$ for $i\in I$, and an
$I\times I$ $\Z$-matrix $(b_{ij})$. We further require that $(b_{ij})$
is skew-symmetrizable via $D$, i.e., $b_{ij}d_{j}=-b_{ji}d_{i}$,
$\forall i,j\in I$. 

Let $\tB$ denote the $I\times I_{\ufv}$-submatrix and $B$ the $I_{\ufv}\times I_{\ufv}$-submatrix.
$(b_{ij})$ is called the $B$-matrix for $t$, $B$ its principal
part, $X_{i}$ its cluster variables on the vertices $i$, and $\{X_{i}|i\in I\}$
its cluster. When $(b_{ij})$ is skew-symmetric, we also associate
to $t$ a unique quiver $\tQ$ without loops or oriented $2$-cycles (called an ice quiver), whose vertex set is $I$
and
\[
b_{ij}=|\{\text{arrows from \ensuremath{i} to \ensuremath{j}}\}|-|\{\text{arrows from \ensuremath{j} to \ensuremath{i}}\}|.
\]
Its full subquiver on $I_{\ufv}$ is denoted by $Q$, called the principal
quiver.

We associate to $t$ a rank-$|I|$ lattice $\Mc(t)=\Z^{I}$ with the
unit vectors $f_{i}$, $i\in I$. Identify the corresponding group
ring $\kk[\Mc(t)]=\oplus_{m\in\Mc(t)}\kk X^{m}$ with the Laurent
polynomial ring $\kk[X_{i}^{\pm}|i\in I]$, such that $X^{f_{i}}=X_{i}$.
We use $\cdot$ to denote the commutative product.

Throughout this paper, we make the following assumption unless otherwise
specified.

\begin{Assumption}[Full rank assumption]\label{assumption:full_rank}

We assume that the matrix $\tB=(b_{ij})_{i\in I,j\in I_{\ufv}}$ is
of full-rank, i.e., its column vectors $\col_{k}\tB$, $k\in I_{\ufv}$,
are linearly independent.

\end{Assumption}

Now consider the quantum case $\kk=\Z[v^{\pm}]$. Any skew-symmetric
$\Z$-matrix $\Lambda=(\Lambda_{ij})_{i,j\in I}$ induces a skew-symmetric
bilinear form $\Lambda$ on $\Z^{I}$ (and thus on $\Mc(t)$) such
that $\Lambda(m,m'):=m^{T}\Lambda m'$, where $(\ )^{T}$ denote the
matrix transpose. By \cite{GekhtmanShapiroVainshtein03,GekhtmanShapiroVainshtein05},
under the full rank assumption, we can choose $\Lambda$ such that
it is compatible with $t$ in the following sense: for any $i\in I$,
$k\in I_{\ufv}$,
\begin{align*}
\Lambda(f_{i},\col_{k}\tB) & =  -\delta_{ik}\diag_{k}
\end{align*}
for some strictly positive integers $\diag_{k}$.

The seed $t$ endowed with a compatible bilinear form $\Lambda$ is
often called a quantum seed.

For a quantum seed $t$, we further endow $\kk[\Mc(t)]$ with the
following twisted product $*$:

\begin{align*}
X^{m}*X^{m'} & =  v^{\Lambda(m,m')}X^{m+m'}.
\end{align*}
Unless otherwise specified, we use the twisted product $*$ for the
multiplication in algebras. The skew-field of fractions for $\kk[\Mc(t)]$
is denoted by $\cF(t)$.

\subsubsection*{Mutations}

We denote the function $[a]_{+}=\max(a,0)$ for any $a\in\R$, and
$[(a_{i})]_{+}=([a_{i}]_{+})$ for any $\R$-vector $(a_{i})$.

Let there be given a seed $t$. For any unfrozen vertex $k\in I_{\ufv}$,
we can construct a new seed $t'=\mu_{k}t$ by an operation $\mu_{k}$
called the mutation in the direction $k$. More precisely, $t'$ consists
of new cluster variables $X_{i}'$ and a matrix $(b_{ij}')$, such
that

\begin{align*}
b_{ij}'  =  \begin{cases}
-b_{ik} & i=k\text{ or }j=k\\
b_{ij}+b_{ik}[b_{kj}]_{+}+[-b_{ik}]_{+}b_{kj} & \text{else}
\end{cases}.
\end{align*}
and $X_{i}'$ are the following elements in the (skew-)field $\cF(t)$:

\begin{align}
X_{i}'  =  \begin{cases}
X_{i} & i\neq k\\
X^{-f_{k}+\sum_{j}[-b_{jk}]_{+}f_{j}}+X^{-f_{k}+\sum_{i}[b_{ik}]_{+}f_{i}} & i=k
\end{cases}.\label{eq:exchange_relation}
\end{align}
Equation \eqref{eq:exchange_relation} is usually called the exchange
relation. Correspondingly, we identify two (skew-)fields $\mu_{k}^{*}:\cF(t')\simeq\cF(t)$,
such that $\mu_{k}^{*}(X_{i}')$ is given by the Laurent polynomial
defined in \eqref{eq:exchange_relation}.

For the quantum case, one can check that (see \cite{BerensteinZelevinsky05}),
$X_{i}'*X_{j}'=v^{2\Lambda'_{ij}}X_{j}'*X_{i}'$ for some $\Z$-matrix
$\Lambda':=(\Lambda'_{ij})_{i,j\in I}$. Moreover, $t'$ becomes a
quantum seed after equipping with the matrix $\Lambda'$.

As before, we associate to $t'$ a lattice $\Mc(t')=\Z^{I}$ with
unit vector $f_{i}'$. We further identify the group ring $\kk[\Mc(t')]=\oplus_{m'\in\Mc(t')}\kk\cdot(X')^{m'}$
with $\kk[(X_{i}')^{\pm}|i\in I]$, such that $(X')^{f_{i}'}=X_{i}'$.
Note that $\Lambda'$ induces a skew-symmetric bilinear form $\Lambda'$
on $\Mc(t')$ such that $\Lambda'(f_{i}',f_{j}')=\Lambda_{ij}'$.

One can check that $\mu_{k}(\mu_{k}t)=t$, see \cite{FominZelevinsky02}\cite{BerensteinZelevinsky05}.

We use $\Delta_{t}^{+}$ to denote the set of seeds obtained from
the initial seed $t$ by iterated mutations. Note that $\Delta_{t}^{+}=\Delta_{\mu_{k}t}^{+}$,
which we will denote by $\Delta^{+}$ for simplicity.

\subsubsection*{Cluster algebras}

Let there be given an initial seed $t_{0}=((X_{i})_{i\in I},(b_{ij}))$
and consider the set of seeds $\Delta^{+}=\Delta_{t_{0}}^{+}$ obtained
from $t_{0}$. The cluster algebra $\bClAlg$ is the $\kk$-algebra
generated by the cluster variables $X_{i}(t)$, $i\in I$, for every
seed $t\in\Delta^{+}$.

The (localized) cluster algebra $\clAlg$ is the localization of $\bClAlg$
at the frozen variables $X_{j}$, $j\in I_{\ufv}$. We denote the
multiplicative group of frozen factors by $\cRing:=\{X^{m}|m\in\Z^{I_{\fv}}\}$,
called the coefficient ring.
\begin{thm}[Laurent phenomenon \cite{FominZelevinsky02}\cite{BerensteinZelevinsky05}]
\label{thm:Laurent_phenomenon}For any $t\in\Delta^{+}$, $\clAlg$
is contained in the Laurent polynomial ring $\kk[\Mc(t)]$.
\end{thm}

\begin{example}[{$\mathfrak{sl}_{3}$ example \cite[Example 8.2.9]{qin2020dual}}]
\label{exa:quantum_cluster_SL3}

Consider the first example in Section \ref{subsec:A-frist-example}.
The set of vertices is $I=\{1,2,3\}$ such that $I_{\ufv}=\{1\}$
and $I_{\fv}=\{2,3\}$. We choose the initial seed $t$ which consists
of the cluster $(X_{1},X_{2},X_{3})$ and the matrix $(b_{ij})=\left(\begin{array}{ccc}
0 & -1 & 1\\
1 & 0 & -1\\
-1 & 1 & 0
\end{array}\right)$.

The mutation at the only unfrozen vertex $1$ generates the new seed
$t'=\mu_{1}t$, which consists of the cluster $(X_{1}',X_{2},X_{3})$
and the matrix $(b_{ij}')=\left(\begin{array}{ccc}
0 & 1 & -1\\
-1 & 0 & 0\\
1 & 0 & 0
\end{array}\right)$. In this example, the set of all seeds is $\Delta^{+}=\{t,t'\}$.

For the classical case $\kk=\Z$, it is straightforward to check
that the exchange relation is $X_{1}\cdot X_{1}'=X_{3}+X_{2}$. For
the quantum case $\kk=\Z[v^{\pm}]$, we can choose $\Lambda=\left(\begin{array}{ccc}
0 & -1 & 1\\
1 & 0 & 0\\
-1 & 0 & 0
\end{array}\right)$. Then the quantized exchange relation is $X_{1}*X_{1}'=vX_{3}+v^{-1}X_{2}$.
\end{example}

\subsection{Upper cluster algebras and cluster varieties\label{subsec:Upper-cluster-algebras}}

Let there be given any initial seed $t_{0}$ and let $\Delta^{+}$
denote the corresponding set of seeds. Identify the skew-fields $\cF(t)$
by \eqref{eq:exchange_relation}, $t\in\Delta^{+}$. The upper cluster
algebra is defined as 
\[
\upClAlg:=\bigcap_{t\in\Delta^{+}}\kk[\Mc(t)].
\]
By the Laurent phenomenon (Theorem \ref{thm:Laurent_phenomenon}),
we have $\clAlg\subset\upClAlg$. It often happens that $\clAlg=\upClAlg$.
Moreover, the upper cluster algebra has the following useful property.
\begin{thm}[Starfish Theorem \cite{BerensteinFominZelevinsky05}]
Under the full rank assumption, for any $t\in\Delta^{+}$, we have
\begin{align*}
\upClAlg & =  \kk[\Mc(t)]\bigcap(\cap_{k\in I_{\ufv}}\kk[\Mc(\mu_{k}t)]).
\end{align*}
 
\end{thm}

Let us describe the geometric counterpart for the upper cluster algebra
$\upClAlg$. Take $\kk=\C$. Then each Laurent polynomial ring $\kk[\Mc(t)]$
is the ring of functions of the split torus $T(t)=(\C^{*})^{I}$.
The mutation map $\mu_{k}^{*}:\cF(\mu_{k}t)\simeq\cF(t)$ induces
a birational map $\mu_{k}:T(t)\dashrightarrow T(\mu_{k}t)$. The cluster
$K_2$ variety\footnote{Following \cite{fock2016symplectic}, we call $\AVar$ the cluster $K_2$-variety because it has a canonical $K_2$-class, see \cite[Section 2.5]{FockGoncharov09a}.} $\AVar$ is defined as the union of tori $\cup_{t\in\Delta^{+}}T(t)$
glued by the mutation maps, and the upper cluster algebra is ring
of functions $\kk[\AVar]$ on $\AVar$.

For completeness, let us sketch the construction of the cluster Poisson
variety $\XVar$ dual to $\AVar$. For each lattice $\Mc(t)=\oplus\Z f_{i}$,
we can construct a dual lattice $N(t)=\oplus\Z e_{i}$ with the pairing
$\langle e_{i},f_{i}\rangle=\frac{1}{d_{i}}$. The Laurent polynomial
ring $\kk[N(t)]=\oplus_{n\in N(t)}\kk Y^{n}$ is the ring of function
of a dual torus $T^{\vee}(t)$. One can define the mutation map $\mu_{k}:T^{\vee}(t)\dashrightarrow T^{\vee}(\mu_{k}t)$
using mutation of $y$-variables (defined in Section \eqref{subsec:Cluster-expansions},
see \cite{FominZelevinsky07} or \cite[(2.3)]{qin2020dual} for the
mutation rule). The corresponding cluster Poisson variety $\XVar$
is the union of the dual tori $\cup_{t\in\Delta^{+}}T^{\vee}(t)$
glued by the mutation maps.

We refer the reader to \cite{gross2013birational} for more details
on cluster varieties.

\subsection{Cluster structure on quantum groups: an $\mathfrak{sl}_{3}$ example\label{subsec:cluster_q_gp_example}}

The quantized coordinate ring associated to the first example in Section
\ref{subsec:A-frist-example} is a quantum cluster algebra. We make
a brief introduction to this algebra and its bases, following the
convention in \cite{Kimura10}.

Consider the Lie algebra $\mathfrak{g}=\mathfrak{sl}_{3}$ and its
Weyl group element $w=w_{0}$, where $w_{0}$ denotes the longest element.
We choose a reduced word $\ow=s_{1}s_{2}s_{1}$ for $w$.

The quantized enveloping algebra $\envAlg(\mathfrak{n}_{-})$ is the
$\C(q)$-algebra generated by $F_{1}$, $F_{2}$ (called Chevalley
generators) satisfying the $q$-Serre relation: 
\begin{align*}
F_{1}^{2}F_{2}-(q+q^{-1})F_{1}F_{2}F_{1}+F_{1}F_{2} & =  0.
\end{align*}

One can construct the following PBW generators for $\envAlg(\mathfrak{n}_{-})$
(called root vectors, see \cite{Lus:intro}):

\begin{align*}
F_{-}(\beta_{1}) & =  F_{1},\\
F_{-}(\beta_{3}) & =  F_{2},\\
F_{-}(\beta_{2}) & =  F_{1}F_{2}-qF_{2}F_{1},
\end{align*}
where we view the vectors $\beta_{i}$ as a formal symbol for simplicity.

For any $\uc\in\N^{3}$ (called Lusztig's parametrization), we construct
the ordered product 
\[
F_{-}(\uc):=F_{-}(\sum c_{i}\beta_{i}):=F_{-}(\beta_{3})^{c_{3}}F_{-}(\beta_{2})^{c_{2}}F_{-}(\beta_{1})^{c_{1}}
\]
where $\sum c_{i}\beta_{i}$ should be viewed as a formal sum. 

The set $\{F_{-}(\uc)|\forall\uc\}$ is a basis for $\envAlg(\mathfrak{n}_{-})$,
called the PBW basis. There is a perfect pairing on $\envAlg(\mathfrak{n}_{-})$,
called Lusztig's bilinear form $(\ ,\ )_{L}$ (alternatively, we can
use Kashiwara's bilinear form). Then we construct the dual PBW basis $\{F_{-}^{\opup}(\uc)|\forall\uc\}$ using the bilinear form $(\ ,\ )_{L}$.
Following \cite{Kimura10}, explicitly, we have 

\begin{align*}
F_{-}^{\opup}(c_{i}\beta_{i}) & =  \prod_{s=1}^{c_{i}}(1-q^{2s})\cdot F_{-}(c_{i}\beta_{i})\\
F_{-}^{\opup}(\uc) & =  F_{-}^{\opup}(\sum c_{3}\beta_{3})F_{-}^{\opup}(\sum c_{2}\beta_{2})F_{-}^{\opup}(\sum c_{1}\beta_{1})
\end{align*}

\begin{defn}
The quantum unipotent subgroup $\qO[N_{-}(w)]$ is the $\Q(q)$-algebra
spanned by the dual PBW basis.
\end{defn}

In this example, let us compute the dual PBW generators for $\qO[N_{-}(w)]$:

\begin{align*}
F_{-}^{\opup}(\beta_{1}) & =  (1-q^{2})F_{1}\\
F_{-}^{\opup}(\beta_{3}) & =  (1-q^{2})F_{2}\\
F_{-}^{\opup}(\beta_{2}) & =  (1-q^{2})(F_{1}F_{2}-qF_{2}F_{1})
\end{align*}

There is the dual bar involution $\sigma$ on $\qO[N_{-}(w)]$, see
\cite{Kimura10}. For example, we have

\begin{align*}
\sigma(F_{i}) & =-q^{2}F_{i},\\
\sigma(F_{1}F_{2}) & =q^{-1}\sigma(F_{2})\sigma(F_{1}).
\end{align*}

Let $<$ denote the lexicographical order on the set of all $\uc\in\N^{3}$. 
\begin{defn}
The dual canonical basis (or upper global basis) $B_{-}^{\opup}=\{B_{-}^{\opup}(\uc)|\forall\uc\}$
for $\qO[N_{-}(w)]$, also denoted by $\dCan$, is the unique basis
subject to the following conditions:
\begin{itemize}
\item $B_{-}^{\opup}(\uc)$ are $\sigma$-invariant;
\item $B_{-}^{\opup}(\uc)\in F_{-}^{\opup}(\uc)+\sum_{\uc'<\uc}q\Z[q]F_{-}^{\opup}(\uc')$.
\end{itemize}
\end{defn}

For example, we can explicitly check that

\begin{align*}
B_{-}^{\opup}(\beta_{i})&= F_{-}^{\opup}(\beta_{i}),\\
B_{-}^{\opup}(\beta_{1}+\beta_{3})&=(1-q^{2})(F_{2}F_{1}-qF_{1}F_{2}).
\end{align*}

Let us define $X_{1}=q^{-1}F_{-}^{\opup}(\beta_{1})$, $X_{1}'=q^{-1}F_{-}^{\opup}(\beta_{3})$,
$X_{2}=q^{-\frac{3}{2}}F_{-}^{\opup}(\beta_{2})$ and $X_{3}=q^{-\frac{3}{2}}B_{-}^{\opup}(\beta_{1}+\beta_{3})$.
It is straightforward to check that they verify the exchange relation
among quantum cluster variables as in Example \ref{exa:quantum_cluster_SL3}:

\begin{align*}
X_{1}X_{1}' & =  q^{\frac{1}{2}}X_{3}+q^{-\frac{1}{2}}X_{2}
\end{align*}
Thus, we have seen that $\qO[N_{-}(w)]$ is the quantum cluster algebra
$\bClAlg$, because they are both generated by these cluster variables.

Let $\qO[N_{-}^{w}]$ denote the localization of $\qO[N_{-}(w)]$
at $X_{2}$, $X_{3}$ (frozen variables). It is called the quantum
unipotent cell. Then $\qO[N_{-}^{w}]$ equals the corresponding (localized)
quantum cluster algebra $\clAlg$.

\section{Topological models for bases}

\label{sec:Topological-models}

In this section, following \cite{thurston2014positive}, we briefly
introduce topological models to provide intuition to readers unfamiliar
with cluster algebras and their bases.

\subsection{Skein algebras and cluster algebras}

A marked surface $\Sigma=(S,M)$ consists of an oriented topological surface
$S$ and a finite set $M$ of marked points in $S$. As a topological
surface, $\Sigma$ is viewed as $S\backslash M$. We require that
each connected component of the boundary $\partial S$ contains at
least one marked point. The marked points contained in the interior
of $S$ are called punctures. 

For simplicity, we make the following assumption. Many results can
be generalized without this assumption.

\begin{Assumption}
We assume that $\Sigma$ has no punctures. 
\end{Assumption}
For technical reasons, we exclude surfaces with empty boundary or containing connected components which are discs with $1$ or $2$ marked points.

We consider curves $C_{i}$ in $\Sigma$ that either end at marked
points or are closed loops. A multicurve (or a diagram) is a finite
union of curves, denoted by $C=\cup C_{i}$. We will consider (multi)curves
up to isotopy fixing the marked points and the crossings. The isotopy
class of $C$ is denoted by $[C]$.

Choose the base ring $\kk=\Z$. The Skein algebra $\Sk(\Sigma)$ is
the quotient of the free $\kk$-module $\oplus_{[C]}\kk[C]$ by the
Skein relations in Figure \ref{fig:Skein-relations}, see \cite{thurston2014positive}.
The multiplication in the Skein algebra is given by the union

\begin{align*}
[C]\cdot[C'] & =  [C\cup C'].
\end{align*}
Note that the empty set provides the multiplicative unit for $\Sk(\Sigma)$.

\begin{figure}
\caption{Skein relations}
\label{fig:Skein-relations}
\subfloat[the crossing resolution]{
\begin{minipage}{
	   0.55\textwidth}
	   \centering
\begin{tikzpicture}[scale=0.80] 

\begin{scope}[xshift=-70,scale=0.2]     
\draw[fill=black!20,dashed] (0,0) circle (4);     
\draw[thick] (-3,-3) to (3,3);     
\draw[thick] (-3,3) to (3,-3); 
\end{scope} 

\node (=) at (-1.25,0) {$=$}; 

\begin{scope}[xshift=0,scale=0.2]     
\draw[fill=black!20,dashed] (0,0) circle (4);     
\draw[thick] (-3,-3) to [out=60,in=-60] (-3,3);     
\draw[thick] (3,-3) to [out=120,in=-120] (3,3); 
\end{scope} 

\node (+) at (1.25,0) {$+$}; 

\begin{scope}[xshift=70,scale=0.2]     
\draw[fill=black!20,dashed] (0,0) circle (4);     
\draw[thick] (-3,-3) to [out=60,in=120] (3,-3);     
\draw[thick] (-3,3) to [out=-60,in=-120] (3,3); 
\end{scope} 

\end{tikzpicture}
\end{minipage}
}
\hfill
\subfloat[the unknot removal]{
\begin{minipage}[c]{
	   0.38\textwidth}
	   \centering
\begin{tikzpicture}[scale=0.80]

\begin{scope}[xshift=-35,scale=0.2]     
\draw[fill=black!20,dashed] (0,0) circle (4);     
\draw[thick] (0,0) circle (2); 
\end{scope} 

\node (=) at (0,0) {$=$}; 

\node (const) at (0.5,0.05) {$-2$}; 

\begin{scope}[xshift=50,scale=0.2]     
\draw[fill=black!20,dashed] (0,0) circle (4); 
\end{scope}  

\end{tikzpicture}
\end{minipage}
}
\subfloat[the monogon removal (for an arc ending at the boundary)]{
\begin{minipage}[c]{
	   0.55\textwidth}
	   \centering
\begin{tikzpicture}[scale=0.80]
\begin{scope}[xshift=-35,scale=0.2] 	
\begin{scope} 	
\clip (0,0) circle (4); 	
\draw[fill=black!20,thick] (-5,-3) to [in=180,out=20] (0,-2) to [in=160,out=0] (5,-3) to [line to] (5,5) to (0,5) to (-5,5); 	
\node (pt) at (0,-2) [marked] {}; 	
\draw[thick] (pt) to [out=60,in=0] (0,2) to [out=180,in=120] (pt); 	
\end{scope} 	
\draw[dashed] (0,0) circle (4); 
\end{scope} 
\node (=) at (0,0) {$=$}; 
\node (0) at (0.5,0.05) {$0$}; 
\end{tikzpicture}
\end{minipage}
}
\end{figure}
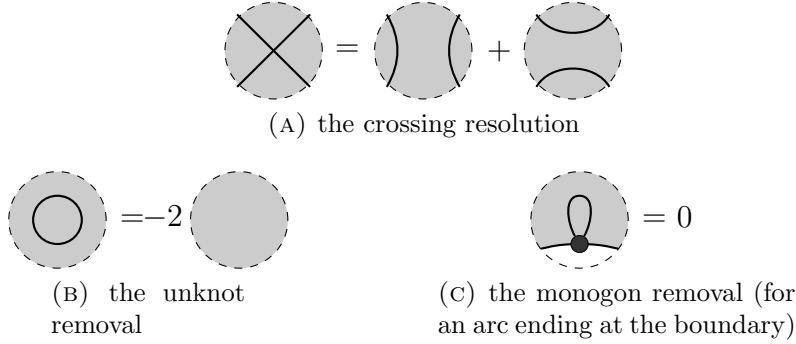

A curve is said to be simple if it is non-contractible and has no
self-crossing. An arc $\gamma$ is a simple curve ending at marked
points. A triangulation\footnote{We should not confuse the symbol $\Delta$ for a triangulation and
$\Delta^{+}$ for the set of seeds.} $\Delta$ is a maximal collection of pairwise non-isotopic non-crossing arcs $\gamma_{i}$
in $\Sigma$. 

Following \cite{FominShapiroThurston08}, we associate a seed $t_{\Delta}$ to $\Delta$ as follows. Associate to each $\gamma_i\in \Delta$ a vertex denoted by $i$. The $i$-th cluster variable of $t_{\Delta}$ is defined to be the Skein algebra element $[\gamma_i]$, which is frozen if $\gamma_i$ is a boundary arc. For each pair of arcs $\gamma_i,\gamma_j\in\Delta$, we define $b_{ij}$
to be the following sum over all triangles $T$ in $\Delta$:\footnote{Here, we view a triangle $T$ as an oriented circle via homotopy, whose orientation is induced by that of $S$, and we investigate the relative positions of the arcs $\gamma_i$ and $\gamma_j$ on the circle.} 
\begin{align} \label{eq:b_matrix}
b_{ij}:=\sum_{T\subset\Delta}\begin{cases}
0 & \mbox{if \ensuremath{\gamma_i} and \ensuremath{\gamma_j} are not both in \ensuremath{T};}\\
1 & \mbox{if \ensuremath{\gamma_j} is the arc immediately clockwise of \ensuremath{\gamma_i} in \ensuremath{T};}\\
-1 & \mbox{if \ensuremath{\gamma_j} is the arc immediately counterclockwise of \ensuremath{\gamma_i} in \ensuremath{T}.}
\end{cases}
\end{align}
See Examples \ref{example:Annulus}.

In this way, we obtain a classical cluster algebra $\bClAlg=\bClAlg(t_{\Delta})$,
which is independent of the triangulation we choose. We skip the details
and refer the reader to \cite{FominShapiroThurston08} for the precise
construction. Such cluster structures are closely related to the Teichmüller
theory for $\Sigma$, see \cite{fomin2018cluster}\cite{FockGoncharov06a}.

For unpunctured $\Sigma$, \cite{muller2016skein} proposed a natural
quantization of $\bClAlg(t_{\Delta})$ and $\Sk(\Sigma)$.
\begin{thm}[\cite{muller2016skein}]
We have natural inclusions of $\kk$-algebras $\bClAlg\subset\Sk(\Sigma)\subset\bUpClAlg$.
When there are at least two marked points on each connected component
of the boundary $\partial S$, we have $\bClAlg=\Sk(\Sigma)=\bUpClAlg$.
\end{thm}

Table \ref{table:compare_topology_cluster} summarizes analogous notions
and structures appearing in the topology of $\Sigma$ and the corresponding
cluster algebra. The symbol $\sim$ means a correspondence or an analogy.

\begin{table}

\caption{Comparison: topology and cluster theory}
\label{table:compare_topology_cluster}%
\begin{tabular}{|c|c|c|}
\hline 
Topology &  & Cluster theory\tabularnewline
\hline 
\hline 
triangulation $\Delta$ & $\sim$ & seed $t_{\Delta}$\tabularnewline
\hline 
arc & $\sim$ & cluster variable\tabularnewline
\hline 
boundary arc & $\sim$ & coefficient/frozen variable\tabularnewline
\hline 
$\cup\gamma_{i}$ for $\gamma_{i}\in\Delta$ & $\sim$ & cluster monomial\tabularnewline
\hline 
union & $\sim$ & multiplication\tabularnewline
\hline 
crossing resolution &  & algebra relation\tabularnewline
\hline 
\end{tabular}
\end{table}

\begin{example}[Annulus]
\label{example:Annulus}

We consider an annulus with two marked points on its boundary, see
Figure \ref{fig:annulus}. Denote its boundary arcs by $b_{1},b_{2}$.
Choose the initial triangulation $\Delta=\{x_{1},x_{2}\}\cup\{b_{1},b_{2}\}$.
By rotating boundary components, it is clear that this marked surface
has infinitely many triangulations, i.e., infinitely many cluster
variables. 

Let us associate vertices $1,2,3,4$ to the arcs $x_1,x_2,b_1,b_2$ respectively. By \eqref{eq:b_matrix}, the matrix $(b_{ij})$ for the initial seed $t_\Delta$ is given by
\begin{align*}
(b_{ij})=\left(\begin{array}{cccc}
0 & -2&1&1\\
2 & 0&-1&-1\\
-1&1&0&0\\
-1&1&0&0
\end{array}\right).
\end{align*}

\begin{figure}
\caption{Annulus}
\label{fig:annulus}	
\begin{tikzpicture}[scale=0.25] 	
\draw[color=blue]  (0,0) circle (10); 	
\draw[color=blue]  (0,0) circle (4); 	
\node[color=blue] (b1) at (10.5,0) {$b_1$}; 	
\node[color=blue] (b2) at (3.5,0) {$b_2$}; 
\draw node[circle,fill,inner sep=0pt,minimum size=1pt] (vA) at (0,10) {A}; 
\draw node[circle,fill,inner sep=0pt,minimum size=1pt] (vB) at (0,4) {B}; 
\draw (0,10) .. controls (4,8) and (8,3) .. node[right]{$x_1$} (8,0); 	
\draw (8,0) .. controls (8,-5) and (8,-8) .. (0,-8); 	
\draw (0,-8) .. controls (-8,-8) and (-8,-5) .. (-8,0); 	
\draw (-8,0) .. controls (-8,3) and (-8,4) .. (0,4); 
\draw (vA) edge node[right]{$x_2$}  (vB);
\end{tikzpicture} 

\end{figure}
\end{example}

By evaluating the frozen variables to $1$ in $\Sk(\Sigma)$, we obtain
the coefficient-free Skein algebra $\Sk'(\Sigma)$. 
\begin{example}[Kronecker type]
\label{example:Kronecker}

We continue Example \ref{example:Annulus}, but removing the frozen
vertices and consider the coefficient-free Skein algebra $\Sk'(\Sigma)$.
Then $\Sk'(\Sigma)$ is the following classical cluster algebra $\bClAlg$
without frozen variables.

Denote the initial cluster variables by $X_i=[x_i]$, $i=1,2$. We have the initial seed $t_{0}=\left((X_{1},X_{2}),B\right)$ whose
$B$-matrix is $B=\left(\begin{array}{cc}
0 & -2\\
2 & 0
\end{array}\right)$. Its quiver is the Kronecker quiver $Q:2\Longrightarrow1$. 

The cluster variables $X_{n},n\in\Z$, are computed recursively via
the exchange relations
\begin{align*}
X_{n+2} & =  X_{n}^{-1}(1+X_{n+1}^{2}).
\end{align*}
The seeds are $t_{n}=((X_{n+1},X_{n+2}),B)$ for $n\in2\Z,$ and $t_{n}=((X_{n+2},X_{n+1}),-B)$,
for $n\in2\Z+1$.

In this example, we have $\bClAlg=\clAlg=\upClAlg=\kk[X_{n}]_{n\in\Z}$.
Its cluster monomials do NOT form a basis.
\end{example}

\subsection{Bangle basis (generic basis)}

By a bangle, we mean a multicurve $C$ without self-crossing or contractible
components. $[C]$ is called a bangle element.

It follows from the Skein relations for $\Sk(\Sigma)$ that any multicurve
$[D]$ can be written as a finite sum of bangle elements $[C]$. We
further have the following result.
\begin{thm}[\cite{musiker2013bases}\cite{thurston2014positive}]
\label{thm:bangle_basis}The bangle elements form a basis for $\Sk(\Sigma)$. 
\end{thm}
The basis is called the bangle basis.

\cite{geiss2020generic} showed that such bangle elements are generic
basis elements for upper cluster algebras, see Section \ref{sec:generic_basis}.
\begin{example}
\label{example:bangle}We continue Example \ref{example:Annulus}. Let
$L$ denote the non-contractible simple loop in the annulus (unique
up to isotopy). The bangle elements are either the cluster monomials,
such as $[x_{1}]^{m_{1}}[x_{2}]^{m_{2}}[b_{1}]^{m_{3}}[b_{2}]^{m_{4}}$,
or the elements $[L]^{m_{1}}[b_{1}]^{m_{2}}[b_{2}]^{m_{3}}$, $m_{i}\geq0$.
See Figure \ref{fig:bangles}.

\begin{figure}
\caption{Bangles}
\label{fig:bangles}	
\begin{tikzpicture}[scale=0.25] 	
\draw[color=blue]  (0,0) circle (10); 	
\draw[color=blue]  (0,0) circle (4); 	
\node[color=blue] (b1) at (10.5,0) {$b_1$}; 	
\node[color=blue] (b2) at (3.5,0) {$b_2$}; 
\draw node[circle,fill,inner sep=0pt,minimum size=1pt] (vA) at (0,10) {A}; 
\draw node[circle,fill,inner sep=0pt,minimum size=1pt] (vB) at (0,4) {B}; 
\draw (0,10) .. controls (4,8) and (8,3) .. node[right]{$x_1$} (8,0); 	
\draw (8,0) .. controls (8,-5) and (8,-8) .. (0,-8); 	
\draw (0,-8) .. controls (-8,-8) and (-8,-5) .. (-8,0); 	
\draw (-8,0) .. controls (-8,3) and (-8,4) .. (0,4); 
\draw (vA) edge node[right]{$x_2$}  (vB);

	\draw[color=red]  (0,0) circle (5); 	
\node[color=red] (b1) at (0,-7) {$L^2$}; 	\draw[color=red]  (0,0) circle (6);
\end{tikzpicture} 
\end{figure}
\end{example}

\subsection{Bracelet basis (theta basis)}

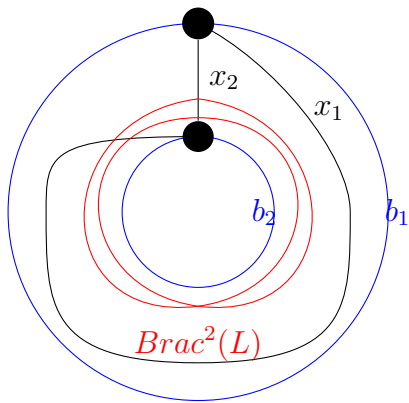
\begin{figure}
\caption{Bracelets}
\label{fig:bracelets}	
\begin{tikzpicture}[scale=0.25] 	
\draw[color=blue]  (0,0) circle (10); 	
\draw[color=blue]  (0,0) circle (4); 	
\node[color=blue] (b1) at (10.5,0) {$b_1$}; 	
\node[color=blue] (b2) at (3.5,0) {$b_2$}; 
\draw node[circle,fill,inner sep=0pt,minimum size=1pt] (vA) at (0,10) {A}; 
\draw node[circle,fill,inner sep=0pt,minimum size=1pt] (vB) at (0,4) {B}; 
\draw (0,10) .. controls (4,8) and (8,3) .. node[right]{$x_1$} (8,0); 	
\draw (8,0) .. controls (8,-5) and (8,-8) .. (0,-8); 	
\draw (0,-8) .. controls (-8,-8) and (-8,-5) .. (-8,0); 	
\draw (-8,0) .. controls (-8,3) and (-8,4) .. (0,4); 
\draw (vA) edge node[right]{$x_2$}  (vB);

\draw[color=red] (0,-5) .. controls (-8,-6) and (-8,5) .. (0,6); 	\draw[color=red] (0,6) .. controls (8,5) and (8,-6) .. (0,-5); 	\draw[color=red] (0,5) .. controls (-7,5) and (-7,-4) .. (0,-5); 	\draw[color=red] (0,-5) .. controls (7,-4) and (7,5) .. (0,5); 	\node at (0,-7) {$\alert{\mathop{Brac}^2(L)}$};

\end{tikzpicture} 

\end{figure}

Assume that $C$ is a bangle. We can denote $C=\cup w_{i}C_{i}$,
$w_{i}\geq0$, such that $C_{i}$ and $C_{j}$ are non-isotopic curves,
and $w_{i}$ denote the multiplicity of $C_{i}$ appearing in $C$
(up to isotopy).

Then, for any closed loop $C_{i}$, we can replace $w_{i}C_{i}$ by
a closed loop having winding number $w_{i}$ and minimal self-crossings ($(w_{i}-1)$-many), denoted
by $\Brac^{w_{i}}(C_{i})$, see Figures \ref{fig:bracelets} and \ref{fig:diagram-permutation}. The new
multicurve $C'$ obtained from $C$ is called a bracelet, and $[C']\in\Sk(\Sigma)$
the bracelet element. The following result easily follows from the
analogous statement for the bangle basis (Theorem \ref{thm:bangle_basis}).
\begin{thm}[\cite{musiker2013bases}\cite{thurston2014positive}]
\label{thm:bracelet_basis}The bracelet elements form a basis for
$\Sk(\Sigma)$.
\end{thm}
The basis is called the bracelet basis.

\begin{example}
We continue Examples \ref{example:Annulus} \label{example:bracelet}.
The bangle elements are either the cluster monomials or the elements
$[\Brac^{m_{1}}(L)][b_{1}]^{m_{2}}[b_{2}]^{m_{3}}$, $m_{i}\geq0$.
See Figure \ref{fig:bangles}.
\end{example}

By \cite{MandelQin2021}, the (quantized) bracelet elements belong
to the (quantized) theta basis for the corresponding upper cluster
algebra.

\subsection{Chebyshev polynomials and bases}

Let us compute the bracelet elements $\Brac^{k}(L)$ in Example \ref{example:bracelet}
explicitly.
\begin{defn}
The Chebyshev polynomials of the first kind are polynomials $T_{k}(z)$,
$k\in\N$, in an indeterminate $z$, subject to the following recursive
definition:

\begin{align*}
T_{0}(z) & =2,\\
T_{1}(z) & =z,\\
T_{k+1}(z) & =zT_{k}(z)-T_{k-1}(z).
\end{align*}
\end{defn}

They have the following properties:

\begin{align*}
T_{k}(z)T_{l}(z) & =T_{k+l}(z)+T_{|k-l|}(z),\ \forall k,l\in\N \\
T_{k}(e^{x}+e^{-x}) & =e^{kx}+e^{-kx},\forall k\in\N.
\end{align*}
In particular, $\forall M\in SL_{2}$, its trace satisfies $\mathrm{tr}(M^{k})=T_{k}(\mathrm{tr}M)$.

\begin{thm}[\cite{musiker2013bases}]
For any $k\in\N$, we have $[\Brac^{k}(L)]=T_{k}([L])$.
\end{thm}

In particular, we see that $[\Brac^{2}(L)]=[L]^{2}-2$ in Figure \ref{fig:bracelets}.

\cite{thurston2014positive} observed that one can associate the trivial
$k$-element permutation to the bangle $kL$ and a non-trivial permutation
for the bracelet $\Brac^{k}(L)$, see Figure \ref{fig:diagram-permutation}. Correspondingly, \cite{thurston2014positive}
proposed considering the formal average of all multicurves corresponding
to any $k$-element permutation. More precisely, $kL$ and $\Brac^{k}L$
only differ by a link diagram contained in a small rectangular neighborhood
(we project the link diagram on the surface and forget its orientation).
And such a link diagram represents a $k$-element permutation. By changing
the link diagrams, we obtain new multicurves, see \cite[Figure 1]{thurston2014positive}. 

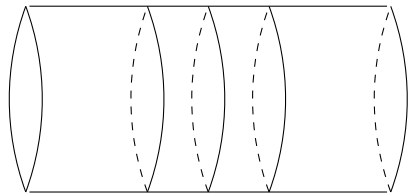
\begin{figure}
\caption{Diagrams and permutations}
\label{fig:diagram-permutation}
\subfloat[A bangle $3L$ represents the trivial permutation.]{
\begin{minipage}{
	   0.8\textwidth}
	   \centering
\begin{tikzpicture}[scale=0.80] 
\draw (5,-1.5) arc (-20:20:4.5);
\draw[dashed] (5,-1.5) arc (-160:-200:4.5);
\draw (-1,-1.5) arc (-160:-200:4.5);
\draw (-1,-1.5) arc (-20:20:4.5);

\draw (3,-1.5) arc (-20:20:4.5);
\draw (2,-1.5) arc (-20:20:4.5);
\draw (1,-1.5) arc (-20:20:4.5);
\draw[dashed] (3,-1.5) arc (-160:-200:4.5);
\draw[dashed] (2,-1.5) arc (-160:-200:4.5);
\draw[dashed] (1,-1.5) arc (-160:-200:4.5);

\node (v1) at (-1.12,1.58) {};
\node (v2) at (5.12,1.58) {};

\node (v3) at (-1.12,-1.5) {};
\node (v4) at (5.12,-1.5) {};

\draw  (v1) edge (v2);
\draw  (v3) edge (v4);

\end{tikzpicture}
\end{minipage}
}
\hfill

\subfloat[A bracelet $\Brac^3(L)$ represents a cyclic permutation.]{
\begin{minipage}[c]{
	   0.8\textwidth}
	   \centering
\begin{tikzpicture}[scale=0.80]
\draw (5,-1.5) arc (-20:20:4.5);
\draw[dashed] (5,-1.5) arc (-160:-200:4.5);
\draw (-1,-1.5) arc (-160:-200:4.5);
\draw (-1,-1.5) arc (-20:20:4.5);

\draw[dashed] (3,-1.5) arc (-160:-200:4.5);
\draw[dashed] (2,-1.5) arc (-160:-200:4.5);
\draw[dashed] (1,-1.5) arc (-160:-200:4.5);

\node (v1) at (-1.12,1.58) {};
\node (v2) at (5.12,1.58) {};

\node (v3) at (-1.12,-1.5) {};
\node (v4) at (5.12,-1.5) {};

\draw  (v1) edge (v2);
\draw  (v3) edge (v4);

\node (v5) at (1,1.58) {};
\node (v6) at (2,1.58) {};
\node (v7) at (3,1.58) {};
\node (v8) at (1,-1.5) {};
\node (v9) at (2,-1.5) {};
\node (v10) at (3,-1.5) {};

\draw  plot[smooth, tension=.7] coordinates {(v8) (2.2,0) (v6)};
\draw  plot[smooth, tension=.7] coordinates {(v9) (3.2,0) (v7)};
\draw  plot[smooth, tension=.7] coordinates {(v10) (3,-0.5)(1.5,0.5) (v5)};
\end{tikzpicture}
\end{minipage}
}
\end{figure}

Let $\Band^{k}(L)$ denote the formal average of the multicurves corresponding
to all $k$-elements permutations and call it a band curve. Let us
compute the band element $[\Band^{k}(L)]$ explicitly.
\begin{defn}
The Chebyshev polynomials of the second kind are polynomials $U_{k}(z)$,
$k\in\N$, in an indeterminate $z$, subject to the following recursive
definition:

\begin{align*}
U_{0}(z) & =  1,\\
U_{1}(z) & =  z,\\
U_{k+1}(z) & =  zU_{k}(z)-U_{k-1}(z).
\end{align*}
\end{defn}

They have the following properties:

\begin{align*}
U_{k}(z)U_{l}(z) & =  U_{k+l}(z)+U_{k+l-2}(z)+\cdots+U_{|k-l|}(z),\ \forall k,l\in\N \\
U_{k}(e^{x}+e^{-x}) & =  e^{kx}+e^{(k-2)x}+\cdots+e^{-kx},\ \forall k\in N.
\end{align*}

\begin{thm}[\cite{thurston2014positive}]
For any $k\in\N$, we have $[\Band^{k}(L)]=U_{k}([L])$.
\end{thm}

For example, we have $[\Band^{2}(L)]=[L]^{2}-1$ in Figure \ref{fig:bands}.

\subsection{Band basis (triangular basis/dual canonical basis)}

\begin{figure}
\caption{Bands}
\label{fig:bands}	
\begin{tikzpicture}[scale=0.25] 	
\draw[color=blue]  (0,0) circle (10); 	
\draw[color=blue]  (0,0) circle (4); 	
\node[color=blue] (b1) at (10.5,0) {$b_1$}; 	
\node[color=blue] (b2) at (3.5,0) {$b_2$}; 
\draw node[circle,fill,inner sep=0pt,minimum size=1pt] (vA) at (0,10) {A}; 
\draw node[circle,fill,inner sep=0pt,minimum size=1pt] (vB) at (0,4) {B}; 
\draw (0,10) .. controls (4,8) and (8,3) .. node[right]{$x_1$} (8,0); 	
\draw (8,0) .. controls (8,-5) and (8,-8) .. (0,-8); 	
\draw (0,-8) .. controls (-8,-8) and (-8,-5) .. (-8,0); 	
\draw (-8,0) .. controls (-8,3) and (-8,4) .. (0,4); 
\draw (vA) edge node[right]{$x_2$}  (vB);

\node at (3,-5) {$\alert{\mathop{Band}^2(L)}=\frac{1}{2}(\mathop{Brac}^2(L)+L^2)$};

\end{tikzpicture} 

\end{figure}

Assume that $C$ is a bangle. We can denote $C=\cup w_{i}C_{i}$,
$w_{i}\geq0$, such that $C_{i}$ and $C_{j}$ are non-isotopic curves,
and $w_{i}$ denotes the multiplicity of $C_{i}$ appearing in $C$
(up to isotopy). Then, for any closed loop $C_{i}$, we replace $w_{i}C_{i}$
by the corresponding band curve $\Band^{w_{i}}(C_{i})$, see \cite[Figure 1]{thurston2014positive},
and obtain a formal average of multicurves $C'$. The corresponding
element $[C']\in\Sk(\Sigma)$, which will be
called a band element, is defined naturally. The following result easily follows from the
analogous statement for the bangle basis (Theorem \ref{thm:bangle_basis}).
\begin{thm}[\cite{thurston2014positive}]
\label{thm:band_basis}The band elements form a basis for $\Sk(\Sigma)$.
\end{thm}
The basis is called the band basis.

\begin{example}
We continue Examples \ref{example:Annulus} and \ref{example:bangle}\label{example:band}.
The band elements are either the cluster monomials or the elements
$[\Band^{m_{1}}(L)][b_{1}]^{m_{2}}[b_{2}]^{m_{3}}$, $m_{i}\geq0$.
See Figure \ref{fig:bangles}.
\end{example}

\begin{conjecture}
\label{conj:band_triangular}The (quantized) band basis coincides
with the common triangular basis in the sense of \cite{qin2017triangular}
(Section \ref{subsec:define_triangular-basis}), after localization
at the frozen variables. 
\end{conjecture}

\begin{rem}
The only known evidence for Conjecture \ref{conj:band_triangular}
is the widely known Example \ref{example:Annulus}, see \cite[Remark 4.22]{thurston2014positive}.
For that case, the common triangular basis is the dual canonical basis
(after a change of frozen variables \cite[Section 4]{qin2020dual}\cite[Section 9]{Qin12}),
and the Chebyshev recursion was found for the dual canonical basis
\cite{Lamp10}.
\end{rem}

\section{Tropical properties and bases\label{sec:Tropical-properties-and-bases}}

\subsection{Cluster expansions and $g$-vectors\label{subsec:Cluster-expansions}}

Let there be given any initial seed $t=((X_{i}),(b_{ij}))$. Let us
define its Laurent monomials $Y_{k}:=X{}^{\col_{k}\tB}$ and $Y^{n}:=X{}^{\tB n}$
for any $n\in\N^{I_{\ufv}}$. $Y_{k}$ are called its $y$-variables.
\begin{thm}[\cite{FominZelevinsky07}\cite{DerksenWeymanZelevinsky09}\cite{Tran09}...\cite{gross2018canonical}]

For any given seed $t'\in\Delta^{+}$, the (quantum) cluster monomial $X(t')^{m}$ in $t'$ for any $m\in\N^{I}$
must have the following Laurent expansion in $\kk[\Mc(t)]$:

\begin{equation}
X(t')^{m}=X{}^{g}\cdot(\sum_{n\in\N^{I_{\ufv}}}c_{n}Y{}^{n})\label{eq:cluster_expansion}
\end{equation}
for some $c_{n}\in\kk$ and $g\in\Mc(t)$. Moreover, $c_{0}=1$.
\end{thm}

Note that, under the full rank assumption, by \eqref{eq:cluster_expansion},
the cluster monomial $X(t')^{m}$ uniquely determines the vector $g$
(called its \textbf{extended $g$-vector}) and the polynomial $F=\sum_{n\in\N^{I_{\ufv}}}c_{n}Y^{n}$
with constant term $1$ (called its \textbf{$F$-polynomial}).

Let $\pr_{I_{\ufv}}:\Z^{I}\rightarrow\Z^{I_{\ufv}}$ denote the natural
projection. $\pr_{I_{\ufv}}g$ is called the principal $g$-vector.
\begin{example}
\label{exa:pointedness}In Example \ref{exa:quantum_cluster_SL3},
we have 

\[
X_{1}'=X_{1}^{-1}\cdot X_{3}\cdot(1+X_{2}\cdot X_{3}^{-1})=X^{(-1,0,1)^{T}}\cdot(1+Y_{1}),
\]
where $\cdot$ denotes the commutative product.
\end{example}

The following conjecture was formulated for the classical case in
\cite{FominZelevinsky02}, and its quantum version is also expected.
\begin{conjecture}[Positivity conjecture ]
\label{conj:positivity}The coefficients $c_{n}$ for cluster monomials
are contained in $\N$ (or in $\N[v^{\pm}]$ for the quantum case).
\end{conjecture}

\begin{rem}
First consider Conjecture \ref{conj:positivity} for the classical
case $\kk=\Z$. Partial results were due to \cite{MusikerSchifflerWilliams09}\cite{HernandezLeclerc09}\cite{Nakajima09}.
The conjecture was verified for skew-symmetric seeds by \cite{LeeSchiffler13}.
It was verified for all skew-symmetrizable seeds by \cite{gross2018canonical}.

Next consider the quantum case $\kk=\Z[v^{\pm}]$. Partial results
were due to \cite{Qin10} \cite{KimuraQin14}. The conjecture was verified
for skew-symmetric seeds by \cite{DavisonMaulikSchuermannSzendroi13}\cite{davison2016positivity}. 

Conjecture \ref{conj:positivity} remains open for skew-symmetrizable
quantum seeds.
\end{rem}

\subsection{Pointedness and dominance orders}

Let there be given any seed $t=((X_{i}),(b_{ij}))$ subject to the
full rank assumption. We are interested in the following type of Laurent
polynomials, which share a similar form with the cluster monomials
in \eqref{eq:cluster_expansion}.
\begin{defn}[Pointedness, dominance order \cite{qin2017triangular}]
Any element $Z\in \kk[\Mc(t)]$ contained in $ X{}^{g}\cdot(1+\Sigma_{n>0}\kk Y{}^{n})$, for some
$g\in\Mc(t)=\Z^{I}$, is said to be $g$-pointed. We define its
degree by $\deg Z=g$.

A subset $\boldsymbol{Z}=\{Z_{g}|g\in\Mc(t)\}\subset\kk[\Mc(t)]$
with $g$-pointed elements $Z_{g}$ is said to be $\Mc(t)$-pointed.
\end{defn}

In order to justify the definition of the degree for $Z$, we need
to introduce the following partial order on $\Mc(t)$ (viewed as the
lattice of Laurent degrees).
\begin{defn}
The dominance order $\prec_{t}$ on $\Mc(t)$ is defined such that we have
$\deg X(t)^{g}Y(t)^{n}\prec_{t}\deg X(t)^{g}$ $\forall n>0$. 

Equivalently, $g'\prec_{t}g$ if and only if there exists some $0\neq n\in\N^{I_{\ufv}}$
such that $g'=g+\tB n$.
\end{defn}

\begin{rem}
\label{rem:pointedness}The definitions of pointedness, degree, and
dominance order are inspired by representation theory. 

More precisely, inspired by the monoidal categorification of cluster
algebras \cite{HernandezLeclerc09} (see Section \ref{subsec:The-common-triangular-in-monoidal}),
we hope that $Z_{g}$ can be compared with the module character $\chi S(w)$
of a highest weight module $S(w)$ for some algebra (such as a quantum
affine algebra \cite{HernandezLeclerc09}). From this point of view, $g$ plays
the role of the highest weight $w$ for $S(w)$. Therefore, the $g$-pointed
element $Z_{g}$ should be parametrized by its degree $g$. 

Our dominance order is inspired by the dominance order for stratification
of graded quiver varieties in \cite{Nakajima09}.
\end{rem}

\subsection{Injective-reachability\label{subsec:Injective-reachability}}
\begin{defn}
A seed $t$ is injective-reachable if there is another seed $t[1]\in\Delta^{+}$
and a permutation $\sigma$ of $I_{\ufv}$, such that its cluster
variables $X_{\sigma k}(t[1])$, $k\in I_{\ufv}$, have the principal
$g$-vector $-f_{k}$ , where $f_{k}$ denote the $k$-th unit vector.
\end{defn}

An (upper) cluster algebra is said to be injective-reachable if its
seeds are. Unless otherwise specified, we make the following assumption
from now on.

\begin{Assumption}\label{assumption:injective-reachable}

The seeds are injective-reachable.

\end{Assumption}

We denote the cluster variable $I_{k}(t):=X_{\sigma k}(t[1])$. For
any $m\in\N^{I_{\ufv}}$, define the ordered product $I(t)^{m}:=v^{\alpha}\prod I_{k}(t)^{m_{k}}$
using the twisted product, where $\alpha\in\Z$ is chosen so that
this is a pointed element.
\begin{rem}
The injective-reachable assumption holds for (almost) all well-known
cluster algebras arising from representation theory or higher Teichm\"uller
theory (except for once-punctured closed surfaces), see \cite{qin2019bases}
for a list.

If a seed $t$ is injective-reachable, so are all seeds in $\Delta^{+}$
\cite{muller2015existence}\cite[Lemma 5.1.1]{qin2017triangular}.

A seed $t$ is injective-reachable if and only if there exists a green
to red sequence\emph{ }in the sense of \cite{keller2011cluster}.

When $t$ is skew-symmetric, one can construct the (completed) Jacobian
algebra $\widehat{J}$ from the corresponding quiver with potential
\cite{DerksenWeymanZelevinsky08,DerksenWeymanZelevinsky09}. Then
$t$ is injective-reachable if and only if the indecomposable injective
modules $I_{k}$ of $\widehat{J}$ can be constructed from the simple
modules by finitely many mutations, whence the name injective-reachable.
\end{rem}

\begin{example}
In Example \ref{exa:pointedness}, we have $\deg X_{1}'=(-1,0,1)^{T}\equiv-f_{1}\mod\Z f_{2}\oplus\Z f_{3}$.
Therefore, the injective-reachable assumption holds.
\end{example}

\subsection{Tropical properties\label{subsec:Tropical-properties}}

Let us define the piecewise linear map $\phi_{k}:\Mc(t)\simeq\Mc(\mu_{k}t)$
such that, for any $g=(g_{i})\in\Mc(t)=\Z^{I}$, its image $g'=(g'_{i})\in\Mc(\mu_{k}t)=\Z^{I}$
is given by

\begin{align*}
g_{i}'  =  \begin{cases}
-g_{k} & i=k\\
g_{i}+[b_{ik}]_{+}g_{k} & i\neq k,\ g_{k}\geq0\\
g_{i}+[-b_{ik}]_{+}g_{k} & i\neq k,\ g_{k}<0
\end{cases}.
\end{align*}
They are called tropical transformations. Then we have the following
result.
\begin{thm}[{\cite[Conjecture 6.10]{FominZelevinsky07}\cite{DerksenWeymanZelevinsky09}\cite{gross2018canonical}}]
If the the extended $g$-vector of a cluster monomial with respect
to the initial seed $t$ is $g$, then $\phi_{k}g$ is its extended
$g$-vector with respect to the initial seed $\mu_{k}t$.
\end{thm}

Denote the tropical semifield $\Z^{T}=(\Z,\max,+)$. We refer the
reader to \cite[Section 2]{gross2018canonical} for the basics of
tropicalization of cluster varieties $\AVar$ and $\XVar$ on $\Z^{T}$.
Let us give the following simplified definition, hiding all geometric
details.
\begin{defn}
The tropical points\footnote{These are the tropical points for the cluster Poisson variety $\XVar$
of the Langlands dual seed \cite{FockGoncharov09}\cite{gross2013birational}.} are the equivalence classes in $\sqcup_{t\in\Delta^{+}}\Mc(t)$,
such that $g\in\Mc(t)$ and $\phi_{k}g\in\Mc(\mu_{k}t)$ are equivalent. We denote the set of tropical points by $\tropMc$.
\end{defn}

Let $[g]$ denote the equivalence class of $g\in\Mc(t)$ in $\tropMc$.
It can be shown that $[g]$ has exactly one representative $g'$ in
$\Mc(t')$ for any $t'\in\Delta^{+}$. Recall that we have a natural identification $\tropMc(t)=\oplus_{i\in I} \Z f_i\simeq\Z^{I}$. Then we obtain $\tropMc\simeq \tropMc(t)\simeq\Z^{I}$ as sets. Note that, while $\tropMc$ does not depend on $t$, the isomorphism $\tropMc\simeq \Z^{I}$ does.
\begin{defn}
An element $Z\in\upClAlg$ is said to be $[g]$-pointed, or parametrized
by $[g]$, if for any $t'\in\Delta^{+}$, $Z$ is $g'$-pointed in
$\kk[\Mc(t')]$, where $g'$ is the representative of $[g]$ in $\Mc(t')$.

A set $Z=\{Z_{[g]}|\forall[g]\in\tropMc\}$ is said to be parametrized
by the tropical points if its elements $Z_{[g]}$ are parametrized
by the tropical points $[g]$.
\end{defn}

\begin{conjecture}[Fock-Goncharov conjecture \cite{FockGoncharov06a,FockGoncharov09}]
\label{conj:FG} The upper cluster algebra $\upClAlg$ possesses
a basis such that it is parametrized by the tropical points.\footnote{\cite{FockGoncharov06a,FockGoncharov09} further expected the basis
to be positive, i.e., its structure constants are non-negative.}
\end{conjecture}

\begin{rem}
The \emph{theta basis $\Theta$} in \cite{gross2018canonical} provides
such a basis for many $\upClAlg$, but fails in the general case.
It was suggested in \cite{gross2018canonical} Conjecture \ref{conj:FG}
should be modified in general.
\end{rem}

\subsection{Bases with good tropical properties\label{subsec:Bases-with-good-tropical}}

We have seen three families of bases via topological models in Section
\ref{sec:Topological-models}. Any of them, if it exists, is known to
be parametrized by the tropical points (and thus verifies Conjecture
\ref{conj:FG}). Therefore, the tropical property in Conjecture \ref{conj:FG}
cannot uniquely determine a basis. 

In general, there exist infinitely many bases parametrized by the
tropical points: see \cite[Theorem 1.2.1]{qin2019bases} for a description
of the infinite space of all such bases. At this moment, the three
families of bases that we have seen in Section \ref{sec:Topological-models}
are the most well-known and interesting ones. See Section \ref{sec:Further-topics}
for a further discussion.

By the following lemma, any basis parametrized by the tropical
points contains all cluster monomials.
\begin{lem}[{\cite[Lemma 3.4.12]{qin2019bases}}]
\label{lem:cluster_monomial_tropical}Assume that the upper cluster algebra
$\upClAlg$ satisfies the injective-reachable assumption. If there
is an element $Z$ in $\upClAlg$ and a cluster monomial $M$, such
that they are both parametrized by the same tropical point, then $Z=M$.
\end{lem}

Moreover, the tropical property provides the following criterion for
verifying a basis.
\begin{thm}[{\cite[Theorem 4.3.1]{qin2019bases}}]
\label{thm:tropical_imply_basis}Assume that the upper cluster algebra
satisfies the injective-reachable assumption. Let $S=\{S_{[g]}|\forall[g]\}$
denote a subset of $\upClAlg$ whose elements $S_{[g]}$ are parametrized
by the tropical points $[g]$. Then $S$ is a basis of $\upClAlg$.
\end{thm}

Theorem \ref{thm:tropical_imply_basis} immediately implies the existence
of the generic basis, see Section \ref{sec:generic_basis}.

\section{The triangular bases from quantum groups and monoidal categories\label{sec:triangular_basis}}

We assume the injective-reachable assumption throughout this section.

\subsection{Triangular basis\label{subsec:define_triangular-basis}}
\begin{defn}[Triangular basis \cite{qin2017triangular}]
\label{def:triangular_basis}Consider the quantum case $\kk=\Z[v^{\pm}]$.
For any seed $t\in\Delta^{+}$, a triangular basis $\can^{t}$ of
$\upClAlg$ is a $\kk$-basis such that
\begin{itemize}
\item $\can^{t}$ contains the quantum cluster monomials in $t,t[1]$.
\item (pointedness) $\can^{t}=\{\can_{g}^{t}|g\in\Mc(t)\}$ such that $\can_{g}^{t}$
are $g$-pointed.
\item (bar-invariance) $\can_{g}^{t}$ are invariant under the bar-involution
$\overline{v^{\alpha}X(t)^{m}}=v^{-\alpha}X(t)^{m}$.
\item (degree triangularity) $\forall i\in I$, $\exists\alpha\in\Z$, such
that
\[
v^{\alpha}X_{i}(t)*\can_{g}^{t}\in\can_{f_{i}+g}^{t}+\sum_{g'\prec_{t}g+f_{i}}v^{-1}\Z[v^{-1}]\can_{g'}^{t}.
\]
\end{itemize}
\end{defn}

Definition \ref{def:triangular_basis} can be generalized for subalgebras\footnote{One might need to restrict to subalgebras of an upper cluster algebra
when Conjecture \ref{conj:FG} fails, see \cite{gross2018canonical}\cite{zhou2020cluster}.} of the upper cluster algebra, see \cite[Section 6.2]{qin2020dual}.
Here we only consider it for upper cluster algebras for simplicity.
\begin{lem}[{\cite[Lemma 6.1.4]{qin2020dual}}]
\label{lem:exist_tri_func}For any $t\in \Delta^+$, the triangular basis $\can^{t}$
is unique if it exists. \label{thm:triangularL_basis_formal_series}
\end{lem}

\begin{proof}
Although we do not have a dual PBW basis for $\upClAlg$, we can construct
an analogous set: for any $g\in\Mc(t)$, we can always construct a
unique $g$-pointed element $\Inj_{g}^{t}$ of the form $v^{\alpha}X^{m}*X(t)^{m'}*I(t)^{m''}$,
for some $\alpha\in\Z$, $m\in\Z^{I_{\fv}}$, $m'=[\pr_{I_{\ufv}}g]_{+}$
and $m''=[-\pr_{I_{\ufv}}g]_{+}$, where $\pr_{I_{\ufv}}g$ is the
principal part of $g$.

The set $\{\Inj_{g}^{t}|\forall g\}$ is NOT a basis for $\upClAlg$.
Nevertheless, with the help of this set and the dominance order, there
is a unique solution $S$ in the formal completion $\kk\left\llbracket \Mc(t)\right\rrbracket :=\kk[\Mc(t)]\otimes_{\kk[Y_{k}|k\in I_{\ufv}]}\kk\llbracket Y_{k}|k\in I_{\ufv}\rrbracket$
via Lusztig's lemma (see \cite[Lemma 8.4]{Nakajima04} \cite[7.10]{Lusztig90}),
such that $S$ satisfies the conditions in Definition \ref{def:triangular_basis}.
See \cite[Theorem 6.1.3]{qin2020dual} for an elementary calculation.
Therefore, if $\can^{t}$ exists, $\can^{t}=S$.
\end{proof}
The $g$-pointed elements $\Inj_{g}^{t}$ in the proof of Lemma \ref{lem:exist_tri_func}
will be called the distinguished functions. They form a topological
basis for the ring of the formal Laurent series $\kk\left\llbracket \Mc(t)\right\rrbracket $,
see \cite[Section 4.1]{qin2019bases}\cite[Section 2.2.2]{davison2019strong}.
\begin{rem}
Unfortunately, we cannot verify the existence of $\can^{t}$ by elementary
calculation, because we only have a candidate in the ring of the formal
Laurent series (see the proof of Lemma \ref{def:triangular_basis}),
and it is hard to tell if the candidate only consists of Laurent polynomials.
\end{rem}

\begin{rem}
The notion of the triangular basis was first proposed by \cite{BerensteinZelevinsky2012}
for \emph{acyclic seeds} (i.e. seeds corresponding to \emph{valued
quivers} without oriented cycles). Their definition is very different
from ours. It was shown in \cite{qin2019compare,qin2020dual} that
their triangular basis equals ours, when the seed is acyclic. 
\end{rem}

\begin{defn}[Common triangular basis]
\label{def:common_triangular_basis} Take $\kk=\Z[v^{\pm}]$. A $\kk$-basis
$\can$ of $\upClAlg$ is said to be the common triangular basis if
$\can$ is the triangular basis $\can^{t}$ for all seeds $t\in\Delta^{+}$.

Note that the common triangular basis in Definition \ref{def:common_triangular_basis}
contains all quantum cluster monomials. By \cite[Proposition 6.4.3]{qin2020dual},
it is parametrized by the tropical points. 
\end{defn}

\begin{rem}
Definition \ref{def:common_triangular_basis} is simpler than but
equivalent to the original definition in \cite{qin2017triangular}, which
imposes the extra condition that $\can$ is parametrized by the tropical
points.
\end{rem}

\subsection{From the dual canonical basis to the common triangular basis\label{subsec:From-dual-canonical-to-triangular}}

\subsubsection*{Cluster structure on quantum groups}

Following the convention of \cite{Kimura10}\cite{kimura2017twist},
we generalize the $\mathfrak{sl}_{3}$ example in Section \ref{subsec:cluster_q_gp_example}
to any Kac-Moody algebra $\frg$ and any element $w$ in its Weyl
group $W$.

As before, we choose any reduced word $\ow$. Then we can construct
the quantum unipotent subgroup $\qO[N_{-}(w)]$ using the dual PBW
basis. It is a quantum analog of the ring of functions $\C[N_{-}(w)]$
for the unipotent subgroup $N_{-}(w)=N_{-}\cap wNw^{-}$. Note that,
for $\frg$ a semi-simple Lie algebra and $w_{0}$ the longest element,
$N_{-}(w)=N_{-}$.

The $q$-center of $\qO[N_{-}(w)]$ is defined as 
\begin{align*}
Z_{q}  =  \{P\in\qO[N_{-}(w)]|\forall x\in\qO[N_{-}(w)],\exists\alpha\in\Z\text{ such that }Px=q^{\alpha}xP\}.
\end{align*}
In fact, it corresponds to the monoid generated by the frozen variables
of the corresponding quantum cluster algebra.

The quantum unipotent cell $\qO[N_{-}^{w}]$ is defined as the localization
of $\qO[N_{-}(w)]$ at the $q$-center. It is a quantum analog of
the ring of functions $\C[N_{-}^{w}]$ for the unipotent cell $N_{-}^{w}=N_{-}\cap BwB$.

It is known that $\qO[N_{-}(w)]$ has the dual canonical basis $\dCan$.
Then the dual canonical basis for $\qO[N_{-}^{w}]$ is defined as
the localization of $\dCan$ at the $q$-center:
\begin{align}
\{q^{\alpha}bP^{-1}|b  \in  \dCan,P\in Z_{q}\},\label{eq:localization}
\end{align}
where the factor $q^{\alpha}$ is chosen such that $q^{\alpha}bP^{-1}$
is invariant under the dual bar involution. See \cite{Kimura10}\cite{kimura2017twist}
for more details.

Following \cite{GeissLeclercSchroeer10,GeissLeclercSchroeer11}\cite{GY13,goodearl2020integral},
we have a natural quantum seed $t_{0}=t_{0}(\ow)$ associated to $\ow$.
\begin{thm}[\cite{GeissLeclercSchroeer10,GeissLeclercSchroeer11}\cite{goodearl2016berenstein,goodearl2020integral}]
\label{thm:quantum_gp_cluster}Consider quantum cluster algebras
for the quantum case $\kk=\Z[v^{\pm}]=\Z[q^{\pm\frac{1}{2}}]$. We
have natural isomorphisms between $\Q(q^{\frac{1}{2}})$-algebras:
\begin{align*}
\qO[N_{-}(w)]\otimes\Q(q^{\frac{1}{2}}) & \simeq  \bClAlg(t_{0}(\ow))\otimes\Q(q^{\frac{1}{2}}),\\
\qO[N_{-}^{w}]\otimes\Q(q^{\frac{1}{2}}) & \simeq  \clAlg(t_{0}(\ow))\otimes\Q(q^{\frac{1}{2}}).
\end{align*}
\end{thm}

We refer the reader to \cite[Section 8]{qin2020dual} for a brief
introduction to the above isomorphisms.

\subsubsection*{Triangular basis on quantum groups}
\begin{thm}[\cite{qin2020dual}]
\label{thm:can_triangular_basis} The dual canonical basis for $\qO[N_{-}^{w}]$
is the common triangular basis for $\clAlg(t_{0}(\ow))$ up to scalar
multiples in $q^{\frac{1}{2}\Z}$. In particular, all quantum cluster
monomials $X(t)^{m}$ belong to the dual canonical basis up to scalar
multiples.
\end{thm}

\begin{rem}[Obstruction]
 It is well-known that the structure constants of $\dCan$ are positive
for a symmetric Kac-Moody algebra $\frg$, the proof of which can be based
on geometric representation theory or monoidal categorification via
quiver Hecke algebras. But this property is NOT true for symmetrizable
$\frg$.

To prove Theorem \ref{thm:can_triangular_basis}, we rely on an analysis
of tropical properties \cite{qin2019bases} instead of the positivity
(or geometric representation theory as in \cite{qin2017triangular},
or categorification). See \cite{qin2020dual} for details.
\end{rem}

\subsubsection*{A comparison between quantum groups and cluster theory}

A comparison of analogous notions and structures in quantum groups
and cluster theory is summarized in Table \ref{table:compare_quantum_groups_cluster},
where $\sim$ means a correspondence or an analogy. Here, to any reduced
word $\ow'$ of $w$, we can associate a seed $t(\ow')\in\Delta^{+}=\Delta_{t_{0}(\ow)}^{+}$.
But the set $\Delta^{+}$ is usually infinite and, in particular,
not every seed arises from $\ow'$. A Lusztig parametrization $\uc$
can be translated into a vector $g\in\Mc(t_{0})$ via a linear map,
and one can compare the (partial) orders on both sides, see \cite[Section 9.1]{qin2020dual}
and \cite{casbi2020dominance}\cite{casbi2019newton}.

\begin{table}
\caption{Comparison: quantum groups and cluster algebras}
\label{table:compare_quantum_groups_cluster}%
\begin{tabular}{|c|c|c|}
\hline 
Quantum groups &  & cluster theory\tabularnewline
\hline 
\hline 
reduced words $\ow'$ for fixed $w$ & $\subset$ & seeds $t\in\Delta^{+}$\tabularnewline
\hline 
Lusztig parametrizations $\uc$ & $\sim$ & $g$-vectors in $\Mc(t_{0})$\tabularnewline
\hline 
lexicographical order &  & dominance order\tabularnewline
\hline 
dual PBW basis $F_{-1}^{\opup}(\uc,\ow)$ & $\neq$ & distinguished functions $I_{g}^{t}$\tabularnewline
\hline 
dual canonical basis & $=$ & common triangular basis $\can$\tabularnewline
\hline 
crystal structure &  & ?\tabularnewline
\hline 
representations &  & \tabularnewline
\hline 
\end{tabular}

\end{table}

\begin{rem}[Generalization of $\dCan$]
Note that the notion of the common triangular basis makes sense for
all injective-reachable quantum cluster algebras, which do not necessarily
arise from quantum groups. In view of Theorem \ref{thm:can_triangular_basis},
the common triangular basis $\can$ is the generalization of the dual
canonical basis in cluster theory.
\end{rem}

\subsection{The common triangular basis in monoidal categories\label{subsec:The-common-triangular-in-monoidal}}

\subsubsection*{Monoidal categorification of cluster algebras}

\cite{HernandezLeclerc09} proposed the notion of monoidal categorification
for classical cluster algebras. Roughly speaking, for a given cluster
algebra $\bClAlg$, one wants to find a monoidal category $(\cC,\otimes)$,
so that its Grothendieck ring $K_{0}(\cC)$ is isomorphic to the cluster
algebra $\bClAlg$. Moreover, the cluster monomials should correspond
to isoclasses of simple modules.
\begin{rem}
Our requirement for monoidal categorification is weaker than the original
proposal by \cite{HernandezLeclerc09}, which demanded that the cluster
monomials are in bijection with all \emph{real simple modules}. It
still remains open if they are in bijection, see the reachability
conjectures \cite[Remark 5.9]{qin2020analog}.
\end{rem}

\cite{HernandezLeclerc09} considered the quantum affine algebra $\envAlg(\widehat{\mathfrak{g}})$
for a simply-laced simple Lie algebra $\mathfrak{g}$ (Dynkin type
$ADE$). Its finite dimensional module category $\mod\envAlg(\widehat{\mathfrak{g}})$
is a monoidal category. For any level $N\in\N$, \cite{HernandezLeclerc09}
considered the level-N monoidal subcategories $\cC_{N}$ of $\mod\envAlg(\widehat{\mathfrak{g}})$.
They showed that $K_{0}(\cC_{N})$ is a classical cluster algebra.
They expected that $\cC_{N}$ provides the monoidal categorification
for this cluster algebra, for which one needs to verify that the cluster
monomials correspond to simple modules.

The $v$-deformation\footnote{It is usually called the $t$-deformed Grothendieck ring $K_{t}(\cC)$ in literature.
We use the quantum parameter $v$ instead of $t$ in our convention.} $K_{v}(\cC)$ of $K_{0}(\cC)$ can be constructed using Nakajima's
quiver varieties \cite{Nakajima04}\cite{VaragnoloVasserot03}, and
it is isomorphic to the corresponding quantum cluster algebra.
\begin{thm}[{\cite[Theorem 1.2.1(II)]{qin2017triangular}}]

The set of (bar-invariant) elements in $K_{v}(\cC_{N})$ corresponding
to the simple modules provides the common triangular basis for the
corresponding quantum cluster algebra (after localization at the frozen
variables, see \eqref{eq:localization}).

In particular, the cluster monomials correspond to simple modules
in $\cC_{N}$, thus verifying (a weaker form of) \cite[Conjecture 13.2]{HernandezLeclerc09}.
\end{thm}

Recently, \cite{kashiwara2020categories} showed that more monoidal
categories consisting of modules of the quantum affine algebra $\envAlg(\widehat{\mathfrak{g}})$
provide monoidal categorification for cluster algebras.

\subsubsection*{Quantum groups and monoidal categorification}

For any symmetric Kac-Moody algebra $\mathfrak{g}$ and any $w\in W$,
one can construct a monoidal category $\cC_{w}$ consisting of $(\Z$-graded)
finite dimensional modules of the corresponding quiver Hecke algebras
\cite{KhovanovLauda08,KhovanovLauda08:III}\cite{Rouquier08}\cite{VaragnoloVasserot09},
such that the quantum unipotent subgroup $\qO[N_{-}(w)]$ is isomorphic
to the Grothendieck ring $K_{0}(\cC_{w})$. See \cite{GeissLeclercSchroeer10}\cite{Kang2018}.

By \cite{VaragnoloVasserot09}, under this isomorphism, the dual canonical
basis $\dCan$ is identified with the set of the isoclasses of the
\emph{self-dual} simple modules in $\cC_{w}$ (see \cite{Kang2018}).
Then Theorem \ref{thm:can_triangular_basis} implies that the common
triangular basis is the set of the isoclasses of the self-dual simple
modules (up to scalar multiplication and localization at the $q$-center,
see \eqref{eq:localization}).

\cite{Kang2018} proposed the notion of quantum monoidal categorification
and proved the following result.
\begin{thm}[\cite{Kang2018}]
The category $\cC_{w}$ provides the quantum monoidal categorification
for the corresponding quantum cluster algebra $\bClAlg(t_{0}(w))$.
In particular, all quantum cluster monomials correspond to self-dual
simple modules in $\cC_{w}$ up to scalar multiples.
\end{thm}

\subsubsection*{Monoidal categorification and cluster theory}

In the literature, for (almost) all cluster algebras known to admit a (quantum)
monoidal categorification, such as those appearing in \cite{qin2017triangular}\cite{Kang2018}\cite{cautis2019cluster},
the set of the isoclasses of (self-dual) simple objects produces the
common triangular basis (up to scalar multiplication and localization at
the frozen variables). This is implied by \cite[Proposition 6.4.3]{qin2020dual}
or \cite[Theorem 6.5.7]{qin2017triangular}. From this point of view, the existence
of the common triangular basis $\can$ suggests that there might exist
a monoidal categorification.

In Remark \ref{rem:pointedness}, we have mentioned the similarity
between the character of a highest weight module $S(w)$ and the Laurent
expansion of a $g$-pointed element. The dominance order in cluster
theory has been studied in \cite{kashiwara2019laurent} from the point of
view of $R$-matrices. Table \ref{table:compare_monoidal_category_cluster}
summarizes the analogous notions and structures in monoidal categories
and cluster theory, when a monoidal categorification is known.

\begin{table}

\caption{Comparison: monoidal categories and cluster theory}
\label{table:compare_monoidal_category_cluster}

\begin{tabular}{|c|c|c|}
\hline 
Monoidal categories $\cC$ &  & cluster theory\tabularnewline
\hline 
\hline 
($v$-deformed) Grothendieck ring & $=$ & (quantum) cluster algebra $\bClAlg$\tabularnewline
\hline 
$\otimes$ &  & $*$\tabularnewline
\hline 
$\oplus$ &  & $+$\tabularnewline
\hline 
homomorphism &  & \tabularnewline
\hline 
$M_{i}\otimes M_{j}\ncong M_{j}\otimes M_{i}$ & $\sim$ & $X_{i}*X_{j}\neq X_{j}*X_{i}$\tabularnewline
\hline 
short exact sequence of ($\Z$-graded) objects & $\supset$ & (quantized) exchange relations\tabularnewline
$vM_{3}\rightarrowtail M_{1}\otimes M_{1}'\twoheadrightarrow v^{-1}M_{2}$ &  & $X_{1}*X_{1}'=vX_{3}+v^{-1}X_{2}$\tabularnewline
\hline 
Rigidity: $\cC$ has left/right dual & $\sim$ & $t$ is Injective-reachable: $\exists\ t[\pm1]$\tabularnewline
\hline 
\cite{casbi2020dominance}\cite{kashiwara2019laurent} & $\sim$ & dominance order\tabularnewline
\hline 
standard modules & $\neq$ & distinguished functions $\Inj_{g}^{t}$\tabularnewline
\hline 
\{(self-dual) simple objects\} & $\sim$ & common triangular basis $\can$\tabularnewline
\hline 
\emph{real} simple object & $\sim$ & cluster monomials\tabularnewline
\hline 
module character & $\sim$ & Laurent expansion\tabularnewline
$\chi S(w)=e^{w}(1+\sum_{n}c_{n}e^{n})$ &  & $\can_{g}=X^{g}(1+\sum c_{n}Y^{n})$\tabularnewline
\hline 
\end{tabular}

\end{table}

\section{The generic basis from representation theory\label{sec:generic_basis}}

Throughout this section, we assume that the principal $B$-matrix
for the initial seed $t_{0}$ is skew-symmetric, so that we can associate
a principal quiver $Q$ to it.

\subsection{Generic cluster characters}

Let $Q$ denote the principal quiver. A potential $W$ is a $\C$-linear
combination of (possibly infinitely many) oriented cycles in $Q$.
The set of potentials is an infinite dimensional space, and we take a generic
one such that $W$ is non-degenerate (see \cite{DerksenWeymanZelevinsky08}).

We define $\hJ$ to be the completed Jacobian algebra associated to
the quiver with potential $(Q,W)$, see \cite{DerksenWeymanZelevinsky08}
for details. In particular, $\hJ$ is a quotient algebra of the completion
of the path algebra $\kk Q$. See Section \ref{subsec:Example:-the-generic}
for an example.

In our convention, we will consider left modules for the opposite
algebra $\hJ^{\op}$. Let $I_{k}$ denote the $k$-th indecomposable
injective module for $\hJ^{\op}$. Under the injective-reachable assumption, each
$I_{k}$ is finite dimensional. Note that the converse is not true,
see the Markov quiver example \cite[Example 4.3]{plamondon2013generic}. 

For any $\hJ^{\op}$-module $V$, consider its minimal injective resolution

\begin{align*}
0  \rightarrow  V\rightarrow I^{m}\rightarrow I^{m'}
\end{align*}
where $m,m'\in\N^{I_{\ufv}}$, and we denote $I^{m}=\oplus I_{k}^{m_{k}}$.
Define the principal $g$-vector for $V$ by $g_{V}=m'-m\in\Z^{I_{\ufv}}$.

Consider the classical case $\kk=\Z$. We have the following Caldero-Chapoton
map (CC-map for short) sending any $\hJ^{\op}$-module $V$ to a Laurent
polynomial in $\kk[\Mc(t_{0})]$.
\begin{defn}[Caldero-Chapoton map \cite{CalderoChapoton06}\cite{DerksenWeymanZelevinsky09}]

For any $\hJ^{\op}$-module $V$, we define $CC(V)$ as the following element in $\kk[\Mc(t_{0})]$:

\begin{align*}
CC(V)  =  X^{g_{V}}\cdot(\sum_{n\in\N^{I_{\ufv}}}\chi(\Gr_{n}V)Y^{n}),
\end{align*}
where $\chi(\ )$ denote the topological Euler characteristic, and
$\Gr_{n}V$ denote the quiver Grassmannian consisting of the $n$-dimensional
submodules of $V$.
\end{defn}

For any vector $g\in\Z^{I_{\ufv}}$, we have the affine space $\Hom_{\hJ^{\op}}(I^{[-g]_{+}},I^{[g]_{+}})$.
Then it has an open dense subset on which $CC(\ker f)$ is constant,
where $f$ belongs to $\Hom_{\hJ^{\op}}(I^{[-g]_{+}},I^{[g]_{+}})$. This constant
value is called the generic cluster character associated to $g$ and
is denoted by $\gen_{g}$. Embedding $\Z^{I_{\ufv}}$ and $\Z^{I_{\fv}}$
into $\Z^{I}$ by adding $0$ entries, we define the generic cluster
character $\gen_{g+m}=X^{m}\cdot\gen_{g}$ for any $g\in\Z^{I_{\ufv}}$
and $m\in\Z^{I_{\fv}}$.

When the generic cluster characters form a basis for $\upClAlg$,
the basis is called the generic basis. In this case, we say that the
generic basis for $\upClAlg$ exists.
\begin{thm}[{\cite[Theorem 1.2.3]{qin2019bases}}]
\label{thm:exist_generic_basis}Consider the classical case $\kk=\Z$.
For any skew-symmetric injective-reachable upper cluster algebra $\upClAlg$
under the full rank assumption, the generic basis exists.
\end{thm}

\begin{proof}
\cite{plamondon2013generic} showed that the generic cluster characters
are parametrized by the tropical points. The desired claim then follows
as a consequence of Theorem \ref{thm:tropical_imply_basis}.
\end{proof}

\subsection{Calculating bases using representation theory: Kronecker quiver\label{subsec:Example:-the-generic}}

Let us discuss the representation theory for the generic basis in
Example \ref{example:Kronecker}. 

In this case, the potential $W$ vanishes, and the completed Jacobian
algebra $\hJ$ is the path algebra $\C Q$. Consider the opposite
quiver $Q^{\op}:1\Longrightarrow2$. Let $\mathrm{Rep}(Q^{\op},\underline{d})$
denote the affine space consisting of the $\underline{d}$-dimensional
$Q^{\op}$-representations over $\mathbb{C}$, which correspond to
$\hJ^{\op}$-modules, see \cite{CrawleyBoevey:quiver}.

Each rigid module $M$ for $\hJ^{\op}=\C Q^{\op}$ corresponds to an
open dense subset in $\mathrm{Rep}(Q^{\op},\underline{d})$. Moreover,
$CC(M)$ is a cluster monomial. It is well-known that all cluster
monomials not divisible by the initial cluster variables take this form
(to restore the initial cluster variables, one needs the notion of decorated
representations \cite{DerksenWeymanZelevinsky09}, $\tau$-rigid
pairs \cite{adachi2014tau}, or cluster categories \cite{BuanMarshReinekeReitenTodorov06} \cite{CalderoChapotonSchiffler06}).

It is known that the points in an open dense subset of $\mathrm{Rep}(Q^{\op},(1,1))$
correspond to the indecomposable modules of dimension $(1,1)$, which
are usually parametrized as $V_{L}(\lambda)$, $\lambda\in\C\mathbb{P}^{1}$.
Choose any such module and denote it by $V_{L}$. Independent of the
choice, we have

\begin{align*}
CC(V_{L}) & =  X_{1}X_{2}^{-1}(1+Y_{2}+Y_{1}Y_{2})
\end{align*}
It is straightforward to check that $CC(V_{L})$ equals the element
$[L]$ in the coefficient-free Skein algebra $\Sk'(\Sigma)$ for the
closed simple loop $L$.

Next, for any $k\in\N$, there is an open dense subset of $\mathrm{Rep}(Q^{\op},(k,k))$,
such that its points correspond to modules isomorphic to $\oplus_{i=1}^{k}V_{L}(\lambda_{i})$,
$\lambda_{i}\neq\lambda_{j}$. Choose any such module and denote it
by $V_{L^{k}}$. Independent of the choice, we have
\begin{align*}
CC(V_{L^{k}})=CC(V_{L})^{k}=[L^{k}]\in\Sk'(\Sigma).
\end{align*}
$CC(V_{L^{k}})$ is a generic cluster character.
\begin{rem}
For any chosen $V_{L}$, we also have a $(k,k)$-dimensional indecomposable
module $V_{\Band^{k}(L)}$, unique up to isomorphism, which
is obtained from $V_{L}$ by iterated self-extensions. It turns out that

\[
CC(V_{\Band^{k}(L)})=[\Band^{k}(L)]\in\Sk'(\Sigma).
\]

We conjecture that this result can be generalized to more surfaces.
\end{rem}

\begin{rem}
In this example, we can compute the bracelets $[\Brac^{k}(L)]$
using transverse quiver Grassmannians instead of quiver Grassmannians
in the CC-map, see \cite{dupont2010transverse,irelli2013homological}. By \cite{allegretti2019categorified}, this formula holds true for unpunctured surfaces.
\end{rem}

\section{Further topics\label{sec:Further-topics}}

\subsection{The theta basis and representation theory}

At first sight, there seems to exist no representation theoretic interpretation for the
theta basis.\footnote{Scattering diagrams and some theta functions can be constructed using representation theory, see \cite{bridgeland2017scattering} \cite{brustle2017stability}.} To see this, let there be given a monoidal categorification
for a cluster algebra. Then the common triangular basis elements should
correspond to the simple modules, which are often viewed as the minimal
building blocks in representation theory, see Section \ref{subsec:The-common-triangular-in-monoidal}.
However, as in the case of Example \ref{example:Kronecker}, one can
expect that a common triangular basis element is sometimes a positive
linear combination of theta basis elements, see Section \ref{sec:Topological-models}.
So the theta basis elements appearing cannot correspond to any genuine
modules.

However, it seems that the theta basis might be understood by considering  the \emph{affine Grassmannians} and the \emph{Langlands
dual}. More precisely, the recent
work \cite{baumann2019mirkovic} defined the Mirković-Vilonen basis
for the algebra $\C[N_{-}]$ for a simple simply-connected algebraic
group $G$. The basis is constructed from cycles on the affine Grassmannian
for the Langlands dual group $G^{\vee}$. By \cite{baumann2019mirkovic},
calculation in small examples seems to suggest
an affirmative answer for the following question.

\begin{Quest}

Does the Mirković-Vilonen basis coincide with the theta basis?

\end{Quest}

\subsection{A new family of bases}

Recall that, by a positive basis, we mean a basis with positive structure
constants.

Assume $\mathfrak{g}$ is a semi-simple Lie algebra for simplicity.
When its Cartan matrix is symmetrizable, we cannot find
a positive basis for the unipotent quantum subgroup $\qO[N_{-}]$ from the previous three families of bases. More precisely, it is known that the
generic basis is often not positive and we do not know its quantization.
The dual canonical basis (the common triangular basis) is not positive
for the symmetrizable case. The theta basis is positive at the classical
level, but not positive at the quantum level for the symmetrizable
case.

Nevertheless, finite-dimensional self-dual simple modules of quiver
Hecke algebras give rise to a positive basis for $\qO[N_{-}]$. Note
that this basis is the same as the dual canonical basis if the Cartan
matrix is symmetric, but different if not.

We hope that this basis is well-behaved for cluster theory. If so,
it provides a new family of interesting bases for cluster algebras.
In particular, we hope to have affirmative answers for the following
questions.

\begin{Quest}

Does this basis contain all cluster monomials? Is it parametrized
by the tropical points?

\end{Quest}


\section*{Acknowledgments}

The author thanks the anonymous referee for carefully reading the manuscript and for giving many comments.

\newcommand{\etalchar}[1]{$^{#1}$}
\def\cprime{$'$}
\providecommand{\bysame}{\leavevmode\hbox to3em{\hrulefill}\thinspace}
\providecommand{\MR}{\relax\ifhmode\unskip\space\fi MR }
\providecommand{\MRhref}[2]{%
  \href{http://www.ams.org/mathscinet-getitem?mr=#1}{#2}
}
\providecommand{\href}[2]{#2}

\bibliography{referenceEprint}

\end{document}